\documentclass[10pt]{amsart}

 \usepackage{amssymb,latexsym,amsmath,amsthm}

\usepackage{fullpage}

\usepackage[numbers,sort&compress]{natbib}

\numberwithin{equation}{section}

\newtheorem{prop}{Proposition}[section]

\newtheorem*{thm*}{Theorem A}

\newtheorem{thm}{Theorem}[section]

\newtheorem{lemma}{Lemma}[section]

\newtheorem{remark}{Remark}[section]

\newtheorem{cor}{Corollary}[section]

\begin{document}

\def\IR{{\mathbb{R}}}

\newcommand{ \Om }{ \Omega}

\newcommand{ \pOm}{\partial \Omega}

\newcommand{ \RO}{\mathbb R^n\setminus \Omega}

\def\RR{{\mathbb{R}}}

\title{Regularity of  extremal solutions of nonlocal elliptic systems}

\author{Mostafa Fazly}

\address{Department of Mathematics, The University of Texas at San Antonio, San Antonio, TX 78249, USA} 

\email{mostafa.fazly@utsa.edu}


\maketitle

\vspace{3mm}

\begin{abstract}  We examine regularity of the extremal solution of nonlinear nonlocal eigenvalue problem
\begin{eqnarray*} 
 \left\{ \begin{array}{lcl}
\hfill   \mathcal L u    &=& \lambda F(u,v) \qquad \text{in} \ \  \Omega,   \\
\hfill \mathcal L v &=& \gamma G(u,v)  \qquad \text{in} \ \   \Omega,  \\
\hfill u,v &=&0 \qquad \qquad \text{on} \ \ \RO , 
\end{array}\right.
  \end{eqnarray*}  
 with an integro-differential operator, including the fractional Laplacian, of the form 
 \begin{equation*}\label{}
\mathcal L(u (x))= \lim_{\epsilon\to 0} \int_{\mathbb R^n\setminus B_\epsilon(x) } [u(x) - u(z)] J(z-x) dz , 
\end{equation*}
 when $J$ is a nonnegative measurable even jump kernel.  In particular, we consider jump kernels of the form of  $J(y)=\frac{a(y/|y|)}{|y|^{n+2s}}$  where $s\in (0,1)$ and $a$ is any nonnegative even measurable function in $L^1(\mathbb {S}^{n-1})$  that satisfies ellipticity assumptions. We first establish stability inequalities for minimal solutions of the above system for a general nonlinearity and  a general kernel. Then, we prove regularity of the extremal solution in dimensions $n < 10s$ and  
$ n<2s+\frac{4s}{p\mp 1}[p+\sqrt{p(p\mp1)}]$ for the Gelfand and Lane-Emden systems when $p>1$ (with positive and negative exponents), respectively.  When $s\to 1$, these dimensions are optimal. However, for the case of $s\in(0,1)$ getting the optimal dimension remains as an open problem.   Moreover, for general nonlinearities,  we consider gradient systems  and we establish regularity of the extremal solution  in dimensions $n<4s$.   As far as we know, this is the first regularity result on the extremal solution of nonlocal system of equations. 
  
\end{abstract}

\noindent
{\it \footnotesize 2010 Mathematics Subject Classification}. {\scriptsize 35R09, 35R11, 35B45, 35B65, 35J50}\\
{\it \footnotesize Key words: Nonlocal elliptic systems, regularity of extremal solutions, stable solutions, nonlinear eigenvalue problems}. {\scriptsize }

\section{Introduction and main results}\label{secin}

Let  $\Omega\subset\mathbb R^n$ be a bounded smooth domain. Consider the nonlinear nonlocal eigenvalue problem 
 \begin{eqnarray*} 
 (P)_{\lambda,\gamma} \qquad
 \left\{ \begin{array}{lcl}
\hfill   \mathcal L u    &=& \lambda F(u,v) \qquad \text{in} \ \  \Omega,   \\
\hfill \mathcal L v &=& \gamma G(u,v)  \qquad \text{in} \ \   \Omega,  \\
\hfill u,v &=&0 \qquad \qquad \text{on} \ \ \RO , 
\end{array}\right.
  \end{eqnarray*}  
where  $ \lambda, \gamma$ are positive parameters, $F,G$ are smooth functions and 
 the operator $\mathcal L$ is an integral operator of convolution type 
  \begin{equation}\label{Lui}
\mathcal L(u (x))= \lim_{\epsilon\to 0} \int_{\mathbb R^n\setminus B_\epsilon(x) } [u(x) - u(z)] J(z-x) dz .
\end{equation}
Here,  $J$ is a nonnegative measurable even jump kernel such that 
\begin{equation}
\int_{\mathbb R^n} \min\{|y|^2,1\} J(y) dy<\infty. 
\end{equation}
The above nonlocal operator with a measurable  kernel 
 \begin{equation}\label{Jump}
J(x,z)= \frac{c(x,z)}{|x-z|^{n+2s}}, 
 \end{equation} 
 when the  function $c(x,z)$ is bounded between two positive constants,  $0<c_1 \le c_2$,   is studied extensively in the literature from both theory of partial differential equations and theory of probability points of view, see the book of Bass \cite{bass} and references therein.    Integro-differential equations and systems,  of the above form,  arise naturally
in the study of stochastic processes with jumps, and more precisely in L\'{e}vy processes. A L\'{e}vy process is a stochastic process with independent and stationary increments.  A special class of such processes is the so called stable
processes. These are the processes that satisfy self-similarity properties, and they
are also the ones appearing in the Generalized Central Limit Theorem. We refer interested readers to the book of Bertoin \cite{ber} for more information.   The infinitesimal generator of any isotropically symmetric stable L\'{e}vy process in $\mathbb R^n$  is 
 \begin{equation}\label{}
\mathcal Lu(x)=\int_{\mathbb S^{n-1}} \int_{-\infty}^\infty [u(x+r\theta)+u(x-r\theta) -2 u(x)]\frac{dr}{r^{1+2s}} d\mu(\theta), 
\end{equation}
where $\mu$ is any nonnegative and finite measure on the unit sphere $\mathbb S^{n-1}$ called the spectral measure and $s\in (0, 1)$. When the spectral measure is absolutely continuous, $d\mu(\theta) = a(\theta)d\theta$, the above operators can be rewritten in the form of 
 \begin{equation}\label{La}
\mathcal L(u (x))= \lim_{\epsilon\to 0} \int_{\mathbb R^n\setminus B_\epsilon(x) } [ u(x+y)+u(x-y)-2u(x)] 
\frac{a(y/|y|)}{|y|^{n+2s}} dy ,
\end{equation}
 where $s\in (0,1)$ and $a$ is any nonnegative even function in $L^1(\mathbb S^{n-1})$. Note that the fractional Laplacian operator $\mathcal L={(-\Delta)}^{s}$ with $0<s<1$ that is 
 \begin{equation}\label{Luj}
\mathcal L u(x) = \lim_{\epsilon\to 0} \int_{\mathbb R^n\setminus B_\epsilon(x) } [u(x) - u(z)]  \frac{c_{n,s}}{|x-z|^{n+2s}} dz,
  \end{equation}   
  for a positive constant $c_{n,s}$ is the simplest stable L\'{e}vy process for $d\mu(\theta) = c_{n,s} d\theta$.   Note that the above operator can be written in the form of (\ref{Lui}) due to the fact that $a$ is even. The regularity of solutions for equation $\mathcal Lu=f$ has been  studied thoroughly in the literature by many experts and in this regard we refer interested to \cite{bass, cs, fro1, rs1, si} and references therein.  The most common assumption on the jump kernel in this context is $0<c_1\le a(\theta) \le c_2$ in $\mathbb S^{n-1}$ and occasionally $a(\theta)\ge c_1>0$ in a subset of  $\mathbb S^{n-1}$ with positive measure.  In this article, we consider the ellipticity assumption on the operator $\mathcal L$  of the form 
\begin{equation}\label{c1c2}
0<c_1 \le \inf_{\nu\in \mathbb S^{n-1}} \int_{\mathbb S^{n-1}} |\nu\cdot\theta|^{2s} a(\theta)d\theta \ \ \text{and} \ \ 0 \le a(\theta) < c_2 \ \ \text{for all}\ \ \theta\in  \mathbb S^{n-1}, 
\end{equation} 
 where $c_1$ and $c_2$ are  constants.  Note that regularity results (interior and boundary) under such an assumption on general operator $\mathcal L$ is studied in the literature,  and in this regard we refer interested readers to \cite{rs2, rs1} and references therein.  For particular nonlinearities of $F$ and $G$, we consider the following Gelfand system  
 \begin{eqnarray*}
(G)_{\lambda,\gamma}\qquad  \left\{ \begin{array}{lcl}
\hfill  \mathcal L  u    &=& \lambda e^v \qquad \text{in} \ \  \Omega,   \\
\hfill \mathcal L v  &=& \gamma e^u \qquad \text{in} \ \  \Omega,  \\
\hfill u, v &=&0 \qquad  \text{on} \ \ \RO , 
\end{array}\right.
  \end{eqnarray*}  
  and the Lane-Emden system, when $p>1$
  \begin{eqnarray*}
(E)_{\lambda,\gamma}\qquad  \left\{ \begin{array}{lcl}
\hfill   \mathcal L u    &=& \lambda (1+v)^p \qquad \text{in} \ \  \Omega,   \\
\hfill \mathcal L v  &=& \gamma (1+u)^p \qquad \text{in} \ \  \Omega,  \\
\hfill u, v &=&0 \qquad \qquad  \text{on} \ \ \RO, 
\end{array}\right.
  \end{eqnarray*}  
  and the  Lane-Emden system with singular nonlinearity, for $p>1$ and when   $0<u,v<1$
  \begin{eqnarray*}
(M)_{\lambda,\gamma}\qquad  \left\{ \begin{array}{lcl}
\hfill  \mathcal  L u    &=& \frac{\lambda}{(1-v)^p} \qquad \text{in} \ \  \Omega,   \\
\hfill\mathcal  L v  &=& \frac{\gamma}{(1-u)^p} \qquad \text{in} \ \  \Omega,  \\
\hfill u, v &=&0 \qquad \qquad  \text{on} \ \ \RO. 
\end{array}\right.
  \end{eqnarray*}  
 Note that for the case of $p=2$ the above singular nonlinearity and system is known as the MicroElectroMechanical Systems (MEMS), see \cite{egg2,lw} and references therein for the mathematical analysis of such equations. In addition, we study the following gradient system with more general nonlinearities 
 \begin{eqnarray*}
(H)_{\lambda,\gamma}\qquad  \left\{ \begin{array}{lcl}
\hfill  \mathcal L  u    &=& \lambda f'(u) g(v) \qquad \text{in} \ \  \Omega,   \\
\hfill \mathcal L  v &=& \gamma f(u) g'(v)  \qquad \text{in} \ \  \Omega,  \\
\hfill u, v &=&0 \qquad  \qquad \text{on} \ \ \RO. 
\end{array}\right.
  \end{eqnarray*}
   The nonlinearities $f$ and $g$ will satisfy various properties but will always at least satisfy
 \begin{equation} \label{R}
f \text{ is smooth, increasing and convex with }  f(0)=1  \text{ and }   f \text{ superlinear at infinity}. 
\end{equation}

A bounded weak solution pair $(u,v)$ is called a classical solution when both components $u,v$ are regular in the interior of $\Omega$ and $(P)_{\lambda,\gamma}$ holds.  Given a  nonlinearity $ f$ which satisfies (\ref{R}), the following nonlinear eigenvalue problem 
 \begin{eqnarray*}
\hbox{$(Q)_{\lambda}$}\hskip 50pt \left\{ \begin{array}{lcl}
\hfill    -\Delta u  &=& \lambda f(u)\qquad \text{in} \ \  \Omega,   \\
\hfill u &=&0 \qquad \qquad  \text{on} \ \ \partial \Omega,
\end{array}\right.
  \end{eqnarray*}
  is  now well-understood. Brezis and V\'{a}zquez in  \cite{BV} raised the question of determining the boundedness of $u^*$ for general nonlinearities $f$  satisfying (\ref{R}).  See, for instance, \cite{BV,Cabre,CC, CR, cro, rs, cdds, Nedev,bcmr,v2,ns} for both local and nonlocal cases.  It is known that there exists a critical parameter  $ \lambda^* \in (0,\infty)$, called the extremal parameter,  such that for all $ 0<\lambda < \lambda^*$ there exists a smooth, minimal solution $u_\lambda$ of $(Q)_\lambda$.  Here the minimal solution means in the pointwise sense.  In addition for each $ x \in \Omega$ the map $ \lambda \mapsto u_\lambda(x)$ is increasing in $ (0,\lambda^*)$.   This allows one to define the pointwise limit $ u^*(x):= \lim_{\lambda \nearrow \lambda^*} u_\lambda(x)$  which can be shown to be a weak solution, in a suitably defined sense, of $(Q)_{\lambda^*}$.  For this reason $ u^*$ is called the extremal solution. It is also known that for $ \lambda >\lambda^*$,  there are no weak solutions of $(Q)_\lambda$.  The regularity of the extremal solution has been of great interests in the literature.  There have several attempts to tackle the problem and here we list a few.    
For a general nonlinearity $f$ satisfying (\ref{R}),  Nedev in \cite{Nedev} proved that $u^*$ is bounded when $ n \le 3$. This was extended to fourth dimensions when $ \Omega$ is a convex domain by Cabr\'{e} in \cite{Cabre}. The convexity of the domain was relaxed by Villegas in \cite{v2}. Most recently, Cabr\'{e} et al. in \cite{cfsr} claimed the regularity result when $n\le 9$.  For the particular nonlinearity  $f(u)=e^u$, known as the Gelfand equation, the regularity is shown $u^*\in L^{\infty}(\Omega)$ for dimensions $n<10$ by Crandall and  Rabinowitz in \cite{CR}, see also \cite{far2}.  If $ \Omega$ is a radial domain in $ \IR^n$ with $ n <10$ the regularity is shown in \cite{CC} when $f$ is a general nonlinearity satisfying conditions (\ref{R}) but without the convexity assumption.   In view of the above result for the exponential nonlinearity, 
this is optimal.  Note that for the case of $\Omega=B_1$, the classification of all radial solutions to this problem was originally done by Liouville in \cite{liou} for $n=2$ and  then in higher dimensions in \cite{ns, jl,MP1} and references therein. For power nonlinearity  $f(u)=(1+u)^p$ and for singular nonlinearity $f(u)=(1-u)^{-p}$ when $0<u<1$ for $p>1$, known as the Lane-Emden equation and MEMS equation respectively, the regularity of extremal solutions is established for the  Joseph-Lundgren  exponent, see \cite{jl}, in the literature.  We refer interested readers to \cite{gg,egg1,egg2,far1,CR} and references therein for regularity results and Liouville theorems.

The regularity of extremal solutions for
nonlocal eigenvalue problem, 
 \begin{eqnarray*}
\hbox{$(S)_{\lambda}$}\hskip 50pt \left\{ \begin{array}{lcl}
\hfill    {(-\Delta)}^s u  &=& \lambda f(u)\qquad \text{in} \ \  \Omega,   \\
\hfill u &=&0 \qquad  \text{on} \ \ \RO,
\end{array}\right.
  \end{eqnarray*}
 is studied in the literature, see \cite{rs,r1, sp, cdds}, when $0<s<1$. However, there are various questions remaining as open problems. Ros-Oton  and Serra in \cite{rs} showed that for a general nonlinearity $f$,  $u^*$ is bounded when $n<4s$. In addition, if  the following limit exists
\begin{equation}\label{conf}
\lim_{t\to\infty} \frac{f(t)f''(t)}{\left|f'(t)\right|^2} <\infty , 
\end{equation}
 then $u^*$ is bounded when $n<10s$. Note that specific nonlinearities $f(u)=e^u$, $f(u)=(1+u)^p$ and $f(u)=(1-u)^{-p}$ for $p>1$ satisfy the above condition (\ref{conf}). When $s\to 1$, the dimension $n<10$ is optimal. However, for the fractional Laplacian $n<10s$  is not optimal, see Remark \ref{rem1}.  In the current article, we prove counterparts of  these regularity results for system of nonlocal equations. Later in \cite{r1}, Ros-Oton considered the fractional Gelfand problem, $f(u)=e^u$, on a domain $\Omega$ that is convex in the $x_i$-direction and symmetric with respect to $\{x_i=0\}$, for $1\le i\le n$. As an example, the unit ball satisfies these conditions. And he proved that $u^*$ is bounded for either $n\le 2s$ or $n>2s$ and 
 \begin{equation}
\frac{ \Gamma(\frac{n}{2}) \Gamma(1+s)}{\Gamma(\frac{n-2s}{2})} >
\frac{ \Gamma^2(\frac{n+2s}{4})}{\Gamma^2(\frac{n-2s}{4})} . 
\end{equation}
This, in particular, implies that $u^*$ is bounded in dimensions $ n\le 7$ for all $s\in (0,1)$.  The above inequality is expected to provide the optimal dimension, see Remark \ref{rem1}. Relaxing the convexity and symmetry conditions on the domain remains an open problem.  Capella et al. in \cite{cdds} studied the extremal solution of a problem related to $(S)_{\lambda}$  in the unit ball $B_1$ with a spectral fractional Laplacian operator that is defined using Dirichlet eigenvalues and eigenfunctions of the Laplacian operator in $B_1$. They showed that $u^*\in L^\infty(B_1)$ when $2\le n< 2\left[s+2+\sqrt{2(s+1)}\right]$. 
 More recently, Sanz-Perela in \cite{sp} proved regularity of the extremal solution of $(S)_{\lambda}$ with the fractional Laplacian operator in the unit ball with the same condition on $n$ and $s$. This implies that $u^*$ is bounded in dimensions $2\le n\le 6$ for all $s\in (0,1)$.  Note also that it is well-known that there is a correspondence between the regularity of stable solutions on bounded domains and the Liouville theorems for stable solutions on the entire space, via rescaling and a blow-up procedure. For the
  classification of solutions of  above nonlocal equations on the entire space we refer interested  readers to \cite{clo,ddw,fw,li}, and for the local equations to  \cite{egg2,far1,far2} and references therein.

 For the case of systems, as discussed in \cite{cf,Mont,faz}, set  $ \mathcal{Q}:=\{ (\lambda,\gamma): \lambda, \gamma >0 \}$ and  define
\begin{equation}
\mathcal{U}:= \left\{ (\lambda,\gamma) \in \mathcal{Q}: \mbox{ there exists a smooth solution $(u,v)$ of $(P)_{\lambda,\gamma}$} \right\}.
\end{equation}
  We assume that $F(0,0),G(0,0)>0$.
 A simple argument shows that if  $F$ is superlinear at infinity for $ u$, uniformly in $v$,  then the set of $ \lambda$ in $\mathcal{U}$ is bounded.  Similarly we assume that $ G$ is superlinear at infinity for  $ v$, uniformly in $u$,  and hence we get $ \mathcal{U}$ is bounded.  We also assume that $F,G$ are increasing in each variable.   This allows the use of a sub/supersolution approach and one easily sees that if $ (\lambda,\gamma) \in \mathcal{U}$ then so is $ (0,\lambda] \times (0,\gamma]$.   One also sees that $ \mathcal{U}$ is nonempty. We now define
 $ \Upsilon:= \partial \mathcal{U} \cap \mathcal{Q}$, which plays the role of the extremal parameter $ \lambda^*$.      Various properties of $ \Upsilon$ are known, see \cite{Mont}.     Given $ (\lambda^*,\gamma^*) \in \Upsilon$ set $ \sigma:= \frac{\gamma^*}{\lambda^*} \in (0,\infty)$ and define
  \begin{equation}
 \Gamma_\sigma:=\{ (\lambda, \lambda \sigma):  \frac{\lambda^*}{2} < \lambda < \lambda^*\}.
   \end{equation}
  We let $ (u_\lambda,v_\lambda)$ denote the minimal solution $(P)_{\lambda, \sigma \lambda}$ for $ \frac{\lambda^*}{2} < \lambda < \lambda^*$.  One easily sees that for each $ x \in \Omega$ that $u_\lambda(x), v_\lambda(x)$ are increasing in $ \lambda$ and hence we define
 \begin{equation}
  u^*(x):= \lim_{\lambda \nearrow \lambda^*} u_\lambda(x) \ \ \text{and} \ \   v^*(x):= \lim_{\lambda \nearrow \lambda^*} v_\lambda(x),
  \end{equation}
   and we call $(u^*,v^*)$ the extremal solution associated with $ (\lambda^*,\gamma^*) \in \Upsilon$.
 Under some very  minor growth assumptions on $F$ and $G$ one can show that $(u^*,v^*)$ is a weak solution of $(P)_{\lambda^*,\gamma^*}$.  For the rest of this article we refer to  $(u^*,v^*)$ as $(u,v)$. For the case of local Laplacian operator,  Cowan and the author in \cite{cf} proved that the extremal solution of $(H)_{\lambda,\gamma}$ when $\Omega$ is a convex domain is regular provided $1\le n \le 3$ for general nonlinearities $f,g\in C^1(\mathbb R)$ that satisfying (\ref{R}).   This can be seen as a counterpart of the Nedev's result for elliptic gradient systems.  For radial solutions,  it is also shown in \cite{cf} that stable solutions are regular in dimensions $1\le n <10$  for general nonlinearities. This is a counterpart of the regularity result of Cabr\'{e}-Capella \cite{CC} and Villegas \cite{v2} for elliptic gradient systems.    For the local Gelfand system,  regularity of the extremal solutions is given by Cowan in \cite{cow} and by Dupaigne et al.  in \cite{dfs} when $n<10$.  
 
Here are our main results. The following theorem deals with regularity of the extremal solution of nonlocal Gelfand, Lane-Emden and MEMS systems. 
\begin{thm}\label{thmg}
Suppose that $ \Omega$ is a bounded smooth  domain in $ \IR^n$. Let $(\lambda^*,\gamma^*) \in \Upsilon$ and $\mathcal L$ is given by (\ref{La}) where the ellipticity condition (\ref{c1c2}) holds and $0<s<1$. Then, the associated extremal solution of  $ (G)_{\lambda^*,\gamma^*}$, $ (E)_{\lambda^*,\gamma^*}$ and  $ (M)_{\lambda^*,\gamma^*}$ is bounded when 
  \begin{eqnarray}\label{dimg}
&&n < 10s, 
\\&&\label{dime} n<2s+\frac{4s}{p-1}[p+\sqrt{p(p-1)}], 
\\&&\label{dimm} n<2s+\frac{4s}{p+1}[p+\sqrt{p(p+1)}], 
 \end{eqnarray}
respectively. 
 \end{thm}
The following theorem is a counterpart of the Nedev regularity result for nonlocal system $(H)_{\lambda,\gamma}$.  
\begin{thm} \label{nedev} Suppose that $ \Omega$ is a bounded smooth convex domain in $ \IR^n$.  Assume that $f$ and $g$ satisfy condition (\ref{R}) and  $f'(0), g'(0)>0$ when  $ f'(\cdot),g'(\cdot)$ are convex and 
\begin{equation} \label{deltaeps}
 \liminf_{s \rightarrow \infty} \frac{\left[f''(s)\right]^2}{f'''(s)f'(s)} >0 \ \ \text{and} \ \   \liminf_{s \rightarrow \infty} \frac{\left[g''(s)\right]^2}{g'''(s)g'(s)} >0,
\end{equation}
 holds.  Let $(\lambda^*,\gamma^*) \in \Upsilon$ and $\mathcal L={(-\Delta)}^s$ for $0<s<1$.  Then,  the associated extremal solution $(u,v)$ of  $ (H)_{\lambda^*,\gamma^*}$ is bounded   where $n < 4s$.  
\end{thm}
In order to prove the above results, we first establish integral estimates for minimal solutions of systems with a general nonlocal operator $\mathcal L$ given by (\ref{Lui}) when $J$ is a nonnegative measurable even jump kernel.   Then, we apply nonlocal Sobolev embedding arguments to conclude boundedness of the extremal solutions. Note that when $\lambda=\gamma$ the systems $(G)_{\lambda,\gamma}$, $(E)_{\lambda,\gamma}$ and $(M)_{\lambda,\gamma}$  turn into scalar equations. 

Here is how this article is structured. In Section \ref{secpre},  we provide regularity theory for nonlocal operators.  In Section \ref{secstab}, we establish various stability inequalities for minimal solutions of systems with a general nonlocal operator of the form (\ref{Lui}) with a nonnegative measurable even jump kernel. In Section \ref{secint}, we provide some technical integral estimates for stable solutions of systems introduced in the above. In Section \ref{secreg}, we apply the integral estimates to establish regularity of extremal solutions for nonlocal  Gelfand, Lane-Emden and MEMS systems with exponential and power nonlinearities. In addition, we provide regularity of the extremal solution for the gradient system $(H)_{\lambda,\gamma}$ with general nonlinearities and also for particular power nonlinearities.

 \section{Preliminaries}\label{secpre}
In this section, we provide regularity results 
not only to the fractional Laplacian, but also to more general integro-differential
equations. We omit the proofs in this section and refer interested readers to corresponding  references.  Let us start with the following  classical regularity result concerning embeddings for the Riesz potential,  see the book of Stein \cite{st}.

\begin{thm}
Suppose that $0<s<1$, $n>2s$ and $f$ and $u$ satisfy 
\begin{equation*}
u    =  {(-\Delta)}^{-s} f \qquad \text{in} \ \  \mathbb R^n,   
  \end{equation*}  
  in the sense that $u$ is the Riesz potential of order $2s$ of $f$. Let $u,f\in L^p(\mathbb R^n)$ when $1\le p<\infty$.  
  \begin{enumerate}
  \item[(i)] For $p=1$, there exists a positive constant $C$ such that 
  \begin{equation*}
  ||u||_{L^q(\mathbb R^n)} \le C ||f||_{L^1(\mathbb R^n)} \ \ \text{for} \ \ q=\frac{n}{n-2s}. 
  \end{equation*}
    \item[(ii)] For $1 < p<\frac{n}{2s}$, there exists a positive constant $C$ such that 
  \begin{equation*}
  ||u||_{L^q(\mathbb R^n)} \le C ||f||_{L^p(\mathbb R^n)} \ \ \text{for} \ \ q=\frac{np}{n-2ps}. 
  \end{equation*}
 \item[(iii)] For $\frac{n}{2s}<p<\infty$, there exists a positive constant $C$ such that 
  \begin{equation*}
  [u]_{C^\beta(\mathbb R^n)} \le C ||f||_{L^p(\mathbb R^n)} \ \ \text{for} \ \  \beta=2s-\frac{n}{p} ,  
  \end{equation*}
  where $[\cdot]_{C^\beta(\mathbb R^n)}$ denotes the $C^\beta$ seminorm. 
  \end{enumerate}
  Here the constant $C$ depending only on $n$, $s$ and $p$. 
\end{thm}

  The above theorem is applied  by  Ros-Oton and Serra in \cite{rs} to establish the following regularity theory  for the fractional Laplacian. See also  \cite{rs1} for the boundary regularity results. 

\begin{prop}
Suppose that $0<s<1$, $n>2s$ and $f\in C(\bar \Omega)$ where $\Omega\subset\mathbb R^n$ is a bounded $C^{1,1}$  domain. Let $u$ be the solution of 
\begin{eqnarray*}
 \left\{ \begin{array}{lcl}
\hfill   {(-\Delta)}^s u    &=& f \qquad \text{in} \ \  \Omega,   \\
\hfill u&=&0 \qquad  \text{in} \ \ \RO. 
\end{array}\right.
  \end{eqnarray*}  
  \begin{enumerate}
  \item[(i)] For $1\le r<\frac{n}{n-2s}$, there exists a positive constant $C$ such that 
  \begin{equation*}
  ||u||_{L^r(\Omega)} \le C ||f||_{L^1(\Omega)} \ \ \text{for} \ \ r<\frac{n}{n-2s}. 
  \end{equation*}
    \item[(ii)] For $1 < p<\frac{n}{2s}$, there exists a positive constant $C$ such that 
  \begin{equation*}
  ||u||_{L^q(\Omega)} \le C ||f||_{L^p(\Omega)} \ \ \text{for} \ \ q=\frac{np}{n-2ps}. 
  \end{equation*}
 \item[(iii)] For $\frac{n}{2s}<p<\infty$, there exists a positive constant $C$ such that 
  \begin{equation*}
  ||u||_{C^\beta(\Omega)} \le C ||f||_{L^p(\Omega)} \ \ \text{for} \ \  \beta=\min\left\{s,2s-\frac{n}{p}\right\}. 
  \end{equation*}
  \end{enumerate}
  Here the constant $C$ depending only on $n$, $s$,  $p$, $r$ and $\Omega$. 
\end{prop}

For the case of $n\le 2s$, the fact that $0<s<1$ implies that $n=1$ and $s\ge \frac{1}{2}$. Note that in this case the Green function $G(x,y)$ is explicitly known. Therefore, $G(\cdot,y)\in L^\infty(\Omega)$ for $s>\frac{1}{2}$ and in $ L^p(\Omega)$ for all $p<\infty$ when $s=\frac{1}{2}$. We summarize this as $||u||_{L^\infty(\Omega)}\le C ||f||_{L^1(\Omega)}$ when $n<2s$. In addition, for the case of $n=2s$, we conclude that 
  $||u||_{L^p(\Omega)}\le C ||f||_{L^1(\Omega)}$ for all $p<\infty$ and $||u||_{L^\infty(\Omega)}\le C ||f||_{L^p(\Omega)}$ for $p>1$.

In what follows we provide a counterpart of the above regularity result for general integro-differential operators given by (\ref{La}). These operators are infinitessimal generators of stable and symmetric L\'{e}vy processes and they  are uniquely determined by a finite measure on the unit
sphere $\mathbb S^{n-1}$, often referred as the spectral measure of the process. When this measure
is absolutely continuous, symmetric stable processes have generators of the
form (\ref{La}) where $0<s<1$ and $a$ is any nonnegative function $L^1(\mathbb S^{n-1})$ satisfying $a(\theta)=a(-\theta)$ for $\theta\in \mathbb S^{n-1}$. The regularity theory for general operators of the form (\ref{La}) has been recently
developed by Fern\'{a}ndez-Real and Ros-Oton in \cite{fro1}. In order to prove this result, authors apply  results of \cite{gh} to study the fundamental solution
associated to the operator $\mathcal L$ in view of the one of the fractional Laplacian.

\begin{prop}\label{propregL}
 Let $\Omega\subset \mathbb R^{n}$ be any bounded domain,
  $0<s<1$ and $f\in L^2(\Omega)$. Let $u$ be any weak solution of 
\begin{eqnarray*}
 \left\{ \begin{array}{lcl}
\hfill   \mathcal L u    &=& f \qquad \text{in} \ \  \Omega,   \\
\hfill u&=&0 \qquad  \text{in} \ \ \RO , 
\end{array}\right.
  \end{eqnarray*}  
  where the operator $\mathcal L$ is given by (\ref{La}) and the ellipticity condition (\ref{c1c2}) holds. Assume that $f\in L^r(\Omega)$ for some $r$. 
  \begin{enumerate}
  \item[(i)] For $1< r<\frac{n}{2s}$, there exists a positive constant $C$ such that 
  \begin{equation*}
  ||u||_{L^q(\Omega)} \le C ||f||_{L^r(\Omega)} \ \ \text{for} \ \ q=\frac{nr}{n-2rs}. 
  \end{equation*}
    \item[(ii)] For $r=\frac{n}{2s}$, there exists a positive constant $C$ such that 
  \begin{equation*}
  ||u||_{L^q(\Omega)} \le C ||f||_{L^r(\Omega)} \ \ \text{for} \ \ q<\infty. 
  \end{equation*}
 \item[(iii)] For $\frac{n}{2s}<r<\infty$, there exists a positive constant $C$ such that 
  \begin{equation*}
  ||u||_{L^\infty(\Omega)} \le C ||f||_{L^r(\Omega)}. 
  \end{equation*}
  \end{enumerate}
  Here the constant $C$ depending only on $n$, $s$,  $r$, $\Omega$ and ellipticity constants. 
\end{prop}
We end this section with this point that $(u,v)$ is a weak solution of $(P)_{\lambda,\gamma}$ for $u,v\in L^1(\Omega)$ if $ F(u,v)\delta^s \in L^1(\Omega)$ and  $G(u,v) \delta^s \in L^1(\Omega)$ where $\delta(x)=\text{dist}(x,\Omega)$ and 
$$ \int_{\Omega} u \mathcal L \zeta  = \int_{\Omega} \lambda F(u,v)  \zeta  \ \ \text{and} \ \   \int_{\Omega} v \mathcal L \eta = \int_{\Omega} \gamma G(u,v)  \eta  , $$
when $\zeta,\eta$ and $\mathcal \zeta,\mathcal \eta$ are bounded in $\Omega$ and $\zeta,\eta\equiv 0$ on $\partial\Omega$.   Any bounded weak solution is a classical solution, in the sense that it is regular in the
interior of $\Omega$,  continuous up to the boundary, and $(P)_{\lambda,\gamma}$ holds pointwise. Note that for the case of local operators,  that is $s = 1$,  the above notion of weak solution is consistent with the one introduced by Brezis et al. in \cite{bcmr,BV}.

\section{Stability inequalities}\label{secstab}

In this section, we provide stability inequalities for minimal solutions of system  $(P)_{\lambda,\gamma}$ for various nonlinearities $F$ and $G$. We start with the following technical lemma in regards to nonlocal operator $\mathcal L$ with even symmetric kernel $J$. 

\begin{lemma}\label{fgprop}
Assume that an operator $\mathcal L$ is given by (\ref{Luj}) with  a measurable symmetric kernel  $J(x,z)=J(x-z)$  that is even.    Then, 
\begin{eqnarray*}
&&\mathcal L(f(x)g(x)) = f(x)\mathcal L(g(x))+g(x)\mathcal L(f(x)) - \int_{\mathbb R^n} \left[f(x)-f(z)  \right] \left[g(x)-g(z)  \right] J(x-z) dz,\\ 
&&\int_{\mathbb R^n} g(x)\mathcal L(f(x)) dx = \frac{1}{2} \int_{\mathbb R^n} \int_{\mathbb R^n}    \left[f(x)-f(z)  \right] \left[g(x)-g(z)  \right] J(x-z) dx dz,  
\end{eqnarray*}
where $f,g\in C^1(\mathbb R^n)$ and the integrals are finite. 
\end{lemma}
\begin{proof}  The proof is elementary and we omit it here. 

\end{proof}

We now establish a stability inequality for minimal solutions of system $(P)_{\lambda,\gamma}$. Note that for the case of local operators this inequality is established by the author and Cowan in \cite{cf} and in \cite{faz}. 

\begin{prop}\label{stablein}  
  Let $(u_\lambda,v_\lambda)$ be a minimal solution of system  $(P)_{\lambda,\gamma}$ such that $u_\lambda,v_\lambda$ are increasing in $\lambda$. Assume that  $J$ is a measurable even kernel and $F_vG_u\ge 0$.  Then,  
\begin{eqnarray} \label{stability}
&& \int_{\Omega}  F_u \zeta^2 +G_v \eta^2 + 2\sqrt{F_vG_u} \zeta\eta dx
\\& \le& \nonumber \frac{1}{2} \int_{\RR^n} \int_{\RR^n}\left( \frac{1}{\lambda} [\zeta(x)- \zeta(z)]^2 +\frac{1}{\gamma} [\eta(x)- \eta(z)]^2\right) J(x-z) dz dx , 
\end{eqnarray} 
for test functions $\zeta,\eta$ so that $\zeta,\eta=0$ in $\RO$. 
\end{prop}  
\begin{proof}
Since $u_\lambda,v_\lambda$ are increasing in $\lambda$,  differentiating $(P)_{\lambda,\gamma}$ with respect to $\lambda$ we get 
 \begin{eqnarray*} 
\hfill   \mathcal L (\partial_\lambda u_\lambda)    &=&  F +  \lambda F_u \partial_\lambda u_\lambda  + \lambda  F_v \partial_\lambda v_\lambda ,  \\
\hfill \mathcal L (\partial_\lambda v_\lambda) &=&\sigma  G + \gamma G_u \partial_\lambda u_\lambda + \gamma G_v \partial_\lambda v_\lambda,  
  \end{eqnarray*}  
  where $u_\lambda,v_\lambda>0$. Multiply both sides with $\frac{1}{\lambda}\frac{\zeta^2}{\partial_\lambda u_\lambda}$ and $\frac{1}{\gamma}\frac{\eta^2}{\partial_\lambda v_\lambda}$ to get  
\begin{eqnarray*} \label{L1}
\frac{1}{\lambda} \mathcal L (\partial_\lambda u_\lambda) \frac{\zeta^2}{\partial_\lambda u_\lambda} + \frac{1}{\gamma}\mathcal  L (\partial_\lambda v_\lambda) \frac{\eta^2}{\partial_\lambda v_\lambda}
 &=&
 \frac{1}{\lambda} F\frac{\zeta^2}{\partial_\lambda u_\lambda} +   F_u  \zeta^2 +   F_v \partial_\lambda v_\lambda  \frac{\zeta^2}{\partial_\lambda u_\lambda}
\\&& +
\frac{\sigma}{\gamma} G\frac{\zeta^2}{\partial_\lambda v_\lambda} +   G_v  \eta^2 +   G_u \partial_\lambda u_\lambda  \frac{\eta^2}{\partial_\lambda v_\lambda} . 
  \end{eqnarray*} 
 Note that the following lower-bound holds for the left-hand side of the above equality
\begin{eqnarray*}
 RHS& \ge &    F_u  \zeta^2 +   G_v  \eta^2 
 +  F_v \partial_\lambda v_\lambda  \frac{\zeta^2}{\partial_\lambda u_\lambda} 
 +  G_u \partial_\lambda u_\lambda  \frac{\eta^2}{\partial_\lambda v_\lambda}
  \ge    F_u  \zeta^2 +   G_v  \eta^2 
 + 2 \sqrt{ F_vG_u} \zeta\eta. 
  \end{eqnarray*}
Integrating the above we end up with 
 \begin{equation} \label{LL1}
\int_{\Omega}  F_u  \zeta^2 +   G_v  \eta^2 + 2 \sqrt{ F_vG_u} \zeta\eta  dx
 \le  \int_{\mathbb R^n}
 \frac{1}{\lambda} L (\partial_\lambda u_\lambda) \frac{\zeta^2}{\partial_\lambda u_\lambda} + \frac{1}{\gamma} L (\partial_\lambda v_\lambda) \frac{\eta^2}{\partial_\lambda v_\lambda} dx . 
   \end{equation} 
Applying Lemma \ref{fgprop}, we have
\begin{eqnarray*}\label{Lphi}
&&\int_{\RR^n} \mathcal L (\partial_\lambda u_\lambda(x)) \frac{\zeta^2(x)}{\partial_\lambda u_\lambda(x)} dx
\\&=&
\frac{1}{2} \int_{\RR^n} \int_{\RR^n} [\partial_\lambda u_\lambda(x) - \partial_\lambda u_\lambda(z)] \left[ \frac{\zeta^2(x)}{\partial_\lambda u_\lambda(x)}-  \frac{\zeta^2(z)}{\partial_\lambda u_\lambda(z)} \right] J(x-z) dx dz. 
\end{eqnarray*}
Note that for $a,b,c,d\in\mathbb R$ when $ab<0$ we have 
\begin{equation*}
(a+b)\left[  \frac{c^2}{a} +  \frac{d^2}{b}  \right] \le (c-d)^2    .
\end{equation*}
Since each $\partial_\lambda u_\lambda$ does not change sign, we have $\partial_\lambda u_\lambda(x)\partial_\lambda u_\lambda(z)>0$. Setting  $a=\partial_\lambda u_\lambda(x)$, $b=-\partial_\lambda u_\lambda(z)$, $c=\zeta(x)$ and  $d=\zeta(z)$ in the above inequality and from the fact that $ab=- \partial_\lambda u_\lambda(x) \partial_\lambda u_\lambda(z)<0$, we  conclude 
\begin{equation*}
 [\partial_\lambda u_\lambda(x) - \partial_\lambda u_\lambda(z)] \left[ \frac{\zeta^2(x)}{\partial_\lambda u_\lambda(x)}-  \frac{\zeta^2(z)}{\partial_\lambda u_\lambda(z)} \right]
  \le [\zeta(x)- \zeta(z)]^2    .
\end{equation*} 
Therefore, 
\begin{equation*}
\int_{\RR^n}\mathcal  L (\partial_\lambda u_\lambda(x)) \frac{\zeta^2(x)}{\partial_\lambda u_\lambda(x)} dx\le  \frac{1}{2} \int_{\RR^n} \int_{\RR^n}  [\zeta(x)- \zeta(z)]^2 J(z-x) dz dx.
\end{equation*} 
This together with (\ref{LL1}) complete the proof. 

\end{proof}

Following ideas provided in the above,  we provide stability inequalities for minimal solutions of  Gelfand,  Lane-Emden and MEMS systems with exponential and power-type nonlinearities. 
 
\begin{cor}\label{stablein}  
  Let $(u,v)$ be the extremal solution of system $(G)_{\lambda,\gamma}$, $(E)_{\lambda,\gamma}$ and $(M)_{\lambda,\gamma}$.  Then,  
\begin{eqnarray}\label{stabilityG}
\sqrt{\lambda\gamma} \int_{\Omega}  e^{\frac{u+v}{2}}\zeta^2 dx &\le&  \frac{1}{2}\int_{\RR^n} \int_{\RR^n} {|\zeta(x)- \zeta(z)|^2} J(x-z) dz dx , 
\\
\label{stabilityE}
p\sqrt{\lambda\gamma} \int_{\Omega} (1+u)^{\frac{p-1}{2}} (1+v)^{\frac{p-1}{2}}  \zeta^2 dx
&\le&  \frac{1}{2}\int_{\RR^n} \int_{\RR^n} {|\zeta(x)- \zeta(z)|^2} J(x-z) dz dx , 
\\ \label{stabilityM}
p\sqrt{\lambda\gamma} \int_{\Omega} (1-u)^{-\frac{p+1}{2}} (1-v)^{-\frac{p+1}{2}}  \zeta^2 dx &\le&  \frac{1}{2}\int_{\RR^n} \int_{\RR^n} {|\zeta(x)- \zeta(z)|^2} J(x-z) dz dx , 
\end{eqnarray} 
for test functions $\zeta$ so that $\zeta=0$ in $\RO$. 
\end{cor}

\begin{cor}\label{stablein1}  
  Let $(u,v)$ be  the extremal solution of system  $(H)_{\lambda,\gamma}$ when $f'g'\ge 0$.  Then,  
\begin{eqnarray}\label{stabilityH}
&& \int_{\Omega}  f''g \zeta^2 +fg'' \eta^2 + 2f'g' \zeta\eta dx
\\& \le&\nonumber  \frac{1}{2} \int_{\RR^n} \int_{\RR^n}\left(  \frac{1}{\lambda} |\zeta(x)- \zeta(z)|^2 + \frac{1}{\gamma}  |\eta(x)- \eta(z)|^2 \right)  J(x-z)  dz dx , 
\end{eqnarray} 
for test functions $\zeta,\eta$ so that $\zeta,\eta=0$ in $\RO$. 
\end{cor}  

\section{Integral estimates for stable solutions}\label{secint}

In this section, we establish some technical integral estimates for stable solutions of systems. Most of the ideas and methods applied in this section are inspired by the ones developed in the literature, see for example \cite{dfs, far1, far2, fg, fw}.  We start with the Gelfand system. 

 \begin{lemma}\label{lemuvg} Suppose that $(u,v)$ is a solution of $ (G)_{\lambda,\gamma}$ when  
 the associated stability inequality (\ref{stabilityG}) holds.  Then, there exists a positive constant $C_{\lambda,\gamma,|\Omega|}=C({\lambda,\gamma,|\Omega|})$ such that 
 \begin{equation*}
 \int_{\Omega} e^{u+v} dx \le C_{\lambda,\gamma,|\Omega|} . 
  \end{equation*}
 \end{lemma}
\begin{proof}  Multiply the second equation of $ (G)_{\lambda,\gamma}$ with $e^{u}-1$  and  integrate to get   
 \begin{equation*}
 \lambda \int_{\Omega} (e^{u}-1) e^v dx= \int_{\Omega} \mathcal L u (e^{u}-1) dx. 
 \end{equation*} 
 From Lemma \ref{fgprop}, we get
 \begin{equation*}
\int_{\Omega} \mathcal L u (e^{u}-1) dx =  \frac{1}{2} \int_{\mathbb R^n} \int_{\mathbb R^n}    \left[u(x)-u(z)  \right] \left[e^{u(x)}-e^{u(z)} \right] J(x-z) dx dz . 
 \end{equation*} 
 Note that for $\alpha,\beta\in\mathbb R$, one can see that 
 \begin{equation*}
 \left| e^{\frac{\beta}{2}} - e^{\frac{\alpha}{2}} \right|^2 \le \frac{1}{4} (e^\beta-e^\alpha)(\beta-\alpha). 
 \end{equation*}
 Applying the above inequality for $\alpha=u(z)$ and $\beta=u(x)$, we obtain 
  \begin{equation*}
 \left| e^{\frac{u(x)}{2}} - e^{\frac{u(z)}{2}} \right|^2 \le \frac{1}{4} (e^{u(x)}-e^{u(z)})(u(x)-u(z)). 
 \end{equation*}
From the above, we conclude 
 \begin{equation*}
 \lambda \int_{\Omega} e^{u+v} dx \ge  \lambda \int_{\Omega} (e^{u}-1) e^v dx \ge  2 \int_{\mathbb R^n} \int_{\mathbb R^n} 
  \left| e^{\frac{u(x)}{2}} - e^{\frac{u(z)}{2}} \right|^2J(x-z) dx dz . 
 \end{equation*}  
Test the stability inequality, Corollary \ref{stablein},  on  $\zeta=e^{\frac{u}{2}}-1$ to get 
  \begin{equation*}\label{}
\sqrt{\lambda\gamma} \int_{\Omega}  e^{\frac{u+v}{2}} (e^{\frac{u}{2}}-1)^2 dx \le  \frac{1}{2}\int_{\RR^n} \int_{\RR^n} \left|e^{\frac{u(x)}{2}}- e^{\frac{u(z)}{2}} \right|^2 J(x-z) dz dx . 
\end{equation*}  
 Combining above inequalities, we conclude 
 \begin{equation}\label{rruv}
\sqrt{\lambda\gamma} \int_{\Omega}  e^{\frac{u+v}{2}} (e^{\frac{u}{2}}-1)^2 dx \le  \frac{1}{4}  \lambda \int_{\Omega} e^{u+v} dx . 
\end{equation}  
 Applying the Young's inequality $e^{u/2}\le \frac{e^u}{4}+1$, we conclude 
 \begin{equation*}\label{}
\int_{\Omega}  e^{\frac{u+v}{2}} e^{\frac{u}{2}} dx \le  \frac{1}{4}   \int_{\Omega} e^{\frac{u+v}{2}}  e^u dx +  \int_{\Omega} e^{\frac{u+v}{2}}  dx . 
\end{equation*}  
From this and  expanding the left-hand side of (\ref{rruv}), we obtain 
 \begin{eqnarray*}
  \lambda \int_{\Omega} e^{u+v} dx + 8 \sqrt{\lambda\gamma} \int_{\Omega}  e^{\frac{u+v}{2}}  dx &\ge& 2 \sqrt{\lambda\gamma}  \int_{\Omega}  e^{\frac{u+v}{2}} e^{u} dx , 
  \\
   \gamma \int_{\Omega} e^{u+v} dx + 8 \sqrt{\lambda\gamma} \int_{\Omega}  e^{\frac{u+v}{2}}  dx  & \ge& 2 \sqrt{\lambda\gamma}  \int_{\Omega}  e^{\frac{u+v}{2}} e^{v} dx  . 
  \end{eqnarray*}
Multiplying these inequalities and applying the Cauchy-Schwarz inequality, i.e.  
\begin{equation*}\label{}
\int_{\Omega}  e^{\frac{u+v}{2}} e^{{u}} dx \int_{\Omega}  e^{\frac{u+v}{2}} e^{{v}} dx  \ge  
\left(\int_{\Omega}   e^{u+v} dx\right)^2, 
\end{equation*}  
we complete the proof. 

 \end{proof} 

We now provide a counterpart of the above estimate for stable solutions of  $ (E)_{\lambda,\gamma}$.   

 \begin{lemma}\label{lemuve} Suppose that $(u,v)$ is a solution of  $ (E)_{\lambda,\gamma}$ when  
 the associated stability inequality (\ref{stabilityE}) holds.  Then, there exists a positive constant $C_{\lambda,\gamma,|\Omega|}=C({\lambda,\gamma,|\Omega|})$ such that 
 \begin{equation}\label{inuvp}
 \int_{\Omega} (1+u)^p(1+v)^p dx \le C_{\lambda,\gamma,|\Omega|}  . 
 \end{equation}
 \end{lemma}
 
 \begin{proof}
 Multiply the second equation of $ (E)_{\lambda,\gamma}$ with $(1+u)^p-1$  and  integrate to get   
 \begin{equation*}
 \lambda \int_{\Omega} [(1+u)^p-1] (1+v)^p dx= \int_{\Omega} \mathcal L u [(1+u)^p-1] dx. 
 \end{equation*} 
 From Lemma \ref{fgprop}, we get
 \begin{equation*}
\int_{\Omega} \mathcal L u [(1+u)^p-1] dx =  \frac{1}{2} \int_{\mathbb R^n} \int_{\mathbb R^n}    \left[u(x)-u(z)  \right] \left[(1+u(x))^p-(1+u(z))^p \right] J(x-z) dx dz . 
 \end{equation*} 
 Note that for $\alpha,\beta\in\mathbb R$, one can see that 
 \begin{equation*}
  [(1+\alpha)^p-(1+\beta)^p](\alpha-\beta) \ge \frac{4p}{(p+1)^2}  \left| (1+\alpha)^{\frac{p+1}{2}} - (1+\beta)^{\frac{p+1}{2}}  \right|^2 .  
 \end{equation*}
 Applying the above inequality for $\alpha=u(x)$ and $\beta=u(z)$, we obtain 
  \begin{equation*}
    [(1+u(x))^p-(1+u(z))^p][u(x)-u(z)] \ge \frac{4p}{(p+1)^2}  \left| (1+u(x))^{\frac{p+1}{2}} - (1+u(z))^{\frac{p+1}{2}}  \right|^2 .  
 \end{equation*}
From the above, we conclude 
 \begin{equation*}
 \lambda \int_{\Omega}  (1+u)^p (1+v)^p dx \ge \frac{4p}{(p+1)^2}  \frac{1}{2}  \int_{\mathbb R^n} \int_{\mathbb R^n} 
   \left| (1+u(x))^{\frac{p+1}{2}} - (1+u(z))^{\frac{p+1}{2}}  \right|^2 J(x-z) dx dz . 
 \end{equation*}  
Test the stability inequality, Corollary \ref{stablein},  on  $\zeta=(1+u)^{\frac{p+1}{2}}-1$ to get 
  \begin{eqnarray*}\label{}
&&\sqrt{\lambda\gamma} p \int_{\Omega} (1+u)^{\frac{p-1}{2}}  (1+v)^{\frac{p-1}{2}} [(1+u)^{\frac{p+1}{2}}-1]^2 dx 
\\&\le&  \frac{1}{2}\int_{\RR^n} \int_{\RR^n} \left| (1+u(x))^{\frac{p+1}{2}}- (1+u(z))^{\frac{p+1}{2}} \right|^2 J(x-z) dz dx . 
\end{eqnarray*}  
 Combining above inequalities, we conclude 
 \begin{equation}\label{rruv}
\sqrt{\lambda\gamma} p \int_{\Omega} (1+u)^{\frac{p-1}{2}}  (1+v)^{\frac{p-1}{2}} [(1+u)^{\frac{p+1}{2}}-1]^2 dx 
\le \frac{(p+1)^2}{4p}  \lambda \int_{\Omega}  (1+u)^p (1+v)^p dx . 
\end{equation}  
Expanding the left-hand side of the inequality and rearranging we get 
\begin{eqnarray}\label{l1}
&&\sqrt{\lambda\gamma} p(1-\epsilon) \int_{\Omega} (1+u)^{\frac{p-1}{2}}  (1+v)^{\frac{p-1}{2}} (1+u)^{p+1} dx 
\\&&\nonumber \le \frac{(p+1)^2}{4p}  \lambda \int_{\Omega}  (1+u)^p (1+v)^p dx 
 +  \frac{ \sqrt{\lambda\gamma}  p}{\epsilon}   \int_{\Omega}  (1+u)^{\frac{p-1}{2}}  (1+v)^{\frac{p-1}{2}}  dx, 
\end{eqnarray}  
where we have used the inequality $a\le \frac{\epsilon}{2}a^2+\frac{1}{2\epsilon}$ for any $\epsilon>0$. Similarly, 
\begin{eqnarray}\label{l2}
&&\sqrt{\lambda\gamma} p(1-\epsilon) \int_{\Omega} (1+v)^{\frac{p-1}{2}}  (1+u)^{\frac{p-1}{2}} (1+v)^{p+1} dx 
\\&&\nonumber \le \frac{(p+1)^2}{4p}  \gamma \int_{\Omega}  (1+u)^p (1+v)^p dx 
 +  \frac{ \sqrt{\lambda\gamma}  p}{\epsilon}   \int_{\Omega}  (1+u)^{\frac{p-1}{2}}  (1+v)^{\frac{p-1}{2}}  dx . 
\end{eqnarray}  
Note that from the Cauchy-Schwarz inequality we get 
\begin{equation*}\label{}
\int_{\Omega} (1+u)^{\frac{p-1}{2}}  (1+v)^{\frac{p-1}{2}} (1+u)^{p+1} dx   \int_{\Omega} (1+v)^{\frac{p-1}{2}}  (1+u)^{\frac{p-1}{2}} (1+v)^{p+1} dx \ge \left( \int_{\Omega}  (1+u)^p (1+v)^p dx  \right)^2 . 
\end{equation*}  
From this and 
multiplying both sides of the above (\ref{l1}) and (\ref{l2}) we conclude 
\begin{equation*}\label{}
\lambda\gamma \left[ p^2(1-\epsilon)^2- \left(\frac{(p+1)^2}{4p}\right)^2\right]\left( \int_{\Omega}  (1+u)^p (1+v)^p dx  \right)^2 \le C_{\epsilon,\lambda,\gamma} \left[   \int_{\Omega}  (1+u)^{\frac{p-1}{2}}  (1+v)^{\frac{p-1}{2}}  dx\right]^2 ,
\end{equation*}  
for small $\epsilon>0$. Note that $p^2- \left(\frac{(p+1)^2}{4p}\right)^2>0$ when $p>1$. Therefore, taking small enough  $\epsilon>0$ and applying the H\"{o}lder's inequality,  we complete the proof.

 \end{proof} 

Here is a counterpart of the above estimate for stable solutions of  $ (M)_{\lambda,\gamma}$.

 \begin{lemma}\label{lemuvm} Suppose that $(u,v)$ is a solution of  $ (M)_{\lambda,\gamma}$ when  
 the associated stability inequality (\ref{stabilityM}) holds.  Then, there exists a positive constant $C_{\lambda,\gamma,|\Omega|}=C({\lambda,\gamma,|\Omega|})$ such that 
 \begin{equation}
 \int_{\Omega} (1-u)^{-p}(1-v)^{-p} dx \le C_{\lambda,\gamma,|\Omega|}. 
 \end{equation}
 \end{lemma}
 
  \begin{proof} The proof is similar to the one provide in Lemma \ref{lemuve}.  Multiply the second equation of $ (M)_{\lambda,\gamma}$ with $(1-u)^{-p}-1$  and  integrate to get   
 \begin{equation*}
 \lambda \int_{\Omega} [(1-u)^{-p}-1] (1-v)^{-p} dx= \int_{\Omega} \mathcal L u [(1-u)^{-p}-1] dx. 
 \end{equation*} 
 From Lemma \ref{fgprop}, we get
 \begin{equation*}
\int_{\Omega} \mathcal L u [(1-u)^{-p}-1] dx =  \frac{1}{2} \int_{\mathbb R^n} \int_{\mathbb R^n}    \left[u(x)-u(z)  \right] \left[(1-u(x))^{-p}-(1-u(z))^{-p} \right] J(x-z) dx dz . 
 \end{equation*} 
 Note that for $\alpha,\beta\in\mathbb R$, one can see that 
 \begin{equation*}
  [(1-\alpha)^{-p}-(1-\beta)^{-p}](\alpha-\beta) \ge \frac{4p}{(p-1)^2}  \left| (1-\alpha)^{\frac{1-p}{2}} - (1-\beta)^{\frac{1-p}{2}}  \right|^2 .  
 \end{equation*}
 Applying the above inequality for $\alpha=u(x)$ and $\beta=u(z)$, we obtain 
  \begin{equation*}
    [(1-u(x))^{-p}-(1-u(z))^{-p}][u(x)-u(z)] \ge \frac{4p}{(p-1)^2}  \left| (1-u(x))^{\frac{-p+1}{2}} - (1-u(z))^{\frac{-p+1}{2}}  \right|^2 .  
 \end{equation*}
From the above, we conclude 
 \begin{equation*}
 \lambda \int_{\Omega}  (1-u)^{-p} (1-v)^{-p} dx \ge \frac{4p}{(p-1)^2}  \frac{1}{2}  \int_{\mathbb R^n} \int_{\mathbb R^n} 
   \left| (1-u(x))^{\frac{-p+1}{2}} - (1-u(z))^{\frac{-p+1}{2}}  \right|^2 J(x-z) dx dz . 
 \end{equation*}  
Test the stability inequality, Corollary \ref{stablein},  on  $\zeta=(1-u)^{\frac{-p+1}{2}}-1$ to get 
  \begin{eqnarray*}\label{}
&&\sqrt{\lambda\gamma} p \int_{\Omega} (1-u)^{-\frac{p+1}{2}}  (1-v)^{-\frac{p+1}{2}} [(1-u)^{\frac{1-p}{2}}-1]^2 dx 
\\&\le&  \frac{1}{2}\int_{\RR^n} \int_{\RR^n} \left| (1-u(x))^{\frac{-p+1}{2}}- (1-u(z))^{\frac{-p+1}{2}} \right|^2 J(x-z) dz dx . 
\end{eqnarray*}  
 Combining above inequalities, we conclude 
 \begin{equation}\label{rruv}
\sqrt{\lambda\gamma} p \int_{\Omega} (1-u)^{-\frac{p+1}{2}}  (1-v)^{-\frac{p+1}{2}} [(1-u)^{\frac{1-p}{2}}-1]^2 dx
\le \frac{(p-1)^2}{4p}  \lambda \int_{\Omega}  (1-u)^{-p }(1-v)^{-p} dx . 
\end{equation}  
Expanding the left-hand side of the inequality and rearranging we get 
\begin{eqnarray}\label{l1}
&&\sqrt{\lambda\gamma} p(1-\epsilon) \int_{\Omega} (1-u)^{-\frac{p+1}{2}}  (1-v)^{-\frac{p+1}{2}} (1-u)^{-p+1} dx 
\\&&\nonumber \le \frac{(p-1)^2}{4p}  \lambda \int_{\Omega}  (1-u)^{-p} (1-v)^{-p} dx 
 +  \frac{ \sqrt{\lambda\gamma}  p}{\epsilon}   \int_{\Omega}  (1-u)^{-\frac{p+1}{2}}  (1-v)^{\frac{p+1}{2}}  dx, 
\end{eqnarray}  
where we have used the inequality $a\le \frac{\epsilon}{2}a^2+\frac{1}{2\epsilon}$ for any $\epsilon>0$. Similarly, 
\begin{eqnarray}\label{l2}
&&\sqrt{\lambda\gamma} p(1-\epsilon) \int_{\Omega} (1-v)^{-\frac{p+1}{2}}  (1-u)^{-\frac{p+1}{2}} (1-v)^{-p+1} dx 
\\&&\nonumber \le \frac{(p-1)^2}{4p}  \gamma \int_{\Omega}  (1-u)^{-p} (1-v)^{-p} dx 
 +  \frac{ \sqrt{\lambda\gamma}  p}{\epsilon}   \int_{\Omega}  (1-u)^{-\frac{p+1}{2}}  (1-v)^{-\frac{p+1}{2}}  dx . 
\end{eqnarray}  
Note that from the Cauchy-Schwarz inequality we get 
\begin{eqnarray*}\label{}
&&\int_{\Omega} (1-u)^{-\frac{p+1}{2}}  (1-v)^{-\frac{p+1}{2}} (1-u)^{-p+1} dx   \int_{\Omega} (1-v)^{-\frac{p+1}{2}}  (1-u)^{-\frac{p+1}{2}} (1-v)^{-p+1} dx \\&& \ge \left( \int_{\Omega}  (1-u)^{-p} (1-v)^{-p} dx  \right)^2  . 
\end{eqnarray*}  
From this and 
multiplying both sides of the above (\ref{l1}) and (\ref{l2}) we conclude 
\begin{equation*}\label{}
\lambda\gamma \left[ p^2(1-\epsilon)^2- \left(\frac{(p-1)^2}{4p}\right)^2\right]\left( \int_{\Omega}  (1-u)^{-p} (1-v)^{-p} dx  \right)^2 \le C_{\epsilon,\lambda,\gamma} \left[   \int_{\Omega}  (1-u)^{-\frac{p+1}{2}}  (1-v)^{-\frac{p+1}{2}}  dx\right]^2 ,
\end{equation*}  
for small $\epsilon>0$. Note that $p^2- (\frac{(p-1)^2}{4p})^2>0$ when $p>0$. Therefore, taking small enough  $\epsilon>0$ and applying the H\"{o}lder's inequality when $p>1$,  we complete the proof. 

 \end{proof} 

In the next lemmata,  we provide integral $L^q(\Omega)$ estimates for Gelfand, Lane-Emden and MEMS systems. We start with the Gelfand system and establish a relation between $\int_{\Omega} e^{\frac{2t+1}{2}u} e^{\frac{v}{2}} dx$ and $\int_{\Omega} e^{\frac{2t+1}{2}v} e^{\frac{u}{2}} dx$ for some constant $t>\frac{1}{2}$. 
 
 \begin{lemma}\label{lemuvgt} Under the same assumptions as Lemma \ref{lemuvg}, set 
 \begin{equation*}
 X:=\int_{\Omega} e^{\frac{2t+1}{2}u} e^{\frac{v}{2}} dx,  Y:=\int_{\Omega} e^{\frac{2t+1}{2}v} e^{\frac{u}{2}} dx, Z:=\int_{\Omega} e^{u} dx , W:= \int_{\Omega} e^{v} dx, 
 \end{equation*}
 where $t>\frac{1}{2}$. Then, 
 \begin{eqnarray}\label{lgXg}
 \begin{array}{lcl} 
\sqrt{\lambda\gamma} X \le (\frac{t }{4}+\epsilon) \lambda X^{\frac{2t-1}{2t}} Y^{\frac{1}{2t}} + C_{\epsilon,\lambda,\gamma,|\Omega|} Z, \\ 
\sqrt{\lambda\gamma} Y \le (\frac{t }{4}+\epsilon)  \gamma Y^{\frac{2t-1}{2t}} X^{\frac{1}{2t}} + C_{\epsilon,\lambda,\gamma,|\Omega|} W,
\end{array} 
 \end{eqnarray}
 where $C_{\epsilon,\lambda,\gamma,|\Omega|} $ is a positive constant. 
 \end{lemma}
   
 \begin{proof} Multiply the second equation of $ (G)_{\lambda,\gamma}$ with $e^{tu}-1$ when $t>\frac{1}{2}$ is a constant. Integrating implies that 
 \begin{equation*}
 \lambda \int_{\Omega} (e^{tu}-1) e^v dx= \int_{\Omega} \mathcal L u (e^{tu}-1) dx. 
 \end{equation*} 
 From Lemma \ref{fgprop}, we get
 \begin{equation*}
\int_{\Omega} \mathcal L u (e^{tu}-1) dx =  \frac{1}{2} \int_{\mathbb R^n} \int_{\mathbb R^n}    \left[u(x)-u(z)  \right] \left[e^{tu(x)}-e^{tu(z)} \right] J(x-z) dx dz . 
 \end{equation*} 
 Note that for $\alpha,\beta\in\mathbb R$, one can see that 
 \begin{equation*}
 \left| e^{\frac{\beta}{2}} - e^{\frac{\alpha}{2}} \right|^2 \le \frac{1}{4} (e^\beta-e^\alpha)(\beta-\alpha). 
 \end{equation*}
 Applying the above inequality for $\alpha=tu(z)$ and $\beta=tu(x)$, we obtain 
  \begin{equation*}
 \left| e^{\frac{tu(x)}{2}} - e^{\frac{tu(z)}{2}} \right|^2 \le \frac{t}{4} (e^{tu(x)}-e^{tu(z)})(u(x)-u(z)). 
 \end{equation*}
From the above, we conclude 
 \begin{equation*}
 \lambda \int_{\Omega} (e^{tu}-1) e^v dx \ge  \frac{2}{t} \int_{\mathbb R^n} \int_{\mathbb R^n} 
  \left| e^{\frac{tu(x)}{2}} - e^{\frac{tu(z)}{2}} \right|^2J(x-z) dx dz . 
 \end{equation*}  
Test the stability inequality, Corollary \ref{stablein},  on  $\zeta=e^{\frac{tu}{2}}-1$ to get 
  \begin{equation*}\label{}
\sqrt{\lambda\gamma} \int_{\Omega}  e^{\frac{u+v}{2}} (e^{\frac{tu}{2}}-1)^2 dx \le  \frac{1}{2}\int_{\RR^n} \int_{\RR^n} \left|e^{\frac{tu(x)}{2}}- e^{\frac{tu(z)}{2}} \right|^2 J(x-z) dz dx . 
\end{equation*}  
 Combining above inequalities, we conclude 
 \begin{equation}\label{uvt2}
\sqrt{\lambda\gamma} \int_{\Omega}  e^{\frac{u+v}{2}} (e^{\frac{tu}{2}}-1)^2 dx \le  \frac{t}{4}  \lambda \int_{\Omega} (e^{tu}-1) e^v dx . 
\end{equation}  
On the other hand,  from the Young inequality we have 
  \begin{equation}\label{uvte1}
\int_{\Omega} e^{\frac{t+1}{2}u} e^{\frac{v}{2}}  dx \le  \frac{\epsilon}{2} \sqrt{\frac{\lambda }{\gamma}} \int_{\Omega} e^{tu} e^{v}  dx  + \frac{1}{2\epsilon} \sqrt{\frac{\gamma }{\lambda}} \int_{\Omega} e^{u}  dx , 
   \end{equation} 
   where $\epsilon$ is a positive constant.  In addition, from the H\"{o}lder inequality we get   
     \begin{equation}\label{uvte2}
   \int_{\Omega}  e^{t u} e^v dx \le \left( \int_{\Omega} e^{\frac{2t+1}{2}u} e^{\frac{v}{2}}  dx \right)^{\frac{2t-1}{2t}} \left( \int_{\Omega} e^{\frac{2t+1}{2}v} e^{\frac{u}{2}}  dx \right)^{\frac{1}{2t}} . 
      \end{equation} 
 Now, expanding both sides of (\ref{uvt2}) we have
    \begin{equation}\label{uvte3}
\sqrt{\lambda\gamma} \int_{\Omega}  e^{\frac{2t+1}{2}u} e^{\frac{v}{2}} dx \le  \frac{t}{4}  \lambda \int_{\Omega} e^{tu} e^v dx + 2 \sqrt{\lambda\gamma} \int_{\Omega}  e^{\frac{t+1}{2}u} e^{\frac{v}{2}} dx . 
\end{equation}  
 Combining (\ref{uvte1}), (\ref{uvte2}) and (\ref{uvte3}) proves the first inequality in (\ref{lgXg}). With similar arguments one can show the second inequality.

    \end{proof}

We now consider the Lane-Emden system and  establish a relation between $\int_{\Omega} (1+u)^{\frac{p+1}{2}+t} (1+v)^{\frac{p-1}{2}} dx$  and $\int_{\Omega} (1+v)^{\frac{p+1}{2}+t} (1+u)^{\frac{p-1}{2}}  dx$ for some constant $t>{1}$.

  \begin{lemma}\label{lemuvet} Under the same assumptions as Lemma \ref{lemuve}, set 
 \begin{eqnarray*}
 &&X:=\int_{\Omega} (1+u)^{\frac{p-1}{2}+t+1} (1+v)^{\frac{p-1}{2}} dx,  \ \ Y:=\int_{\Omega} (1+v)^{\frac{p-1}{2}+t+1} (1+u)^{\frac{p-1}{2}}  dx, 
 \\&&Z:=\int_{\Omega} (1+u)^{\frac{p-1}{2}}(1+v)^{\frac{p-1}{2}}  dx , 
 \end{eqnarray*}
 for $t>1$. Then, for some constant $0<\epsilon<1$ we get 
 \begin{eqnarray}\label{lgXet}
 \begin{array}{lcl} 
\sqrt{\lambda\gamma} p(1-\epsilon) X \le \frac{(t+1)^2 }{4t}  \lambda X^{\frac{2t-p+1}{2(t+1)}} Y^{\frac{p+1}{2(t+1)}} +  C_{\epsilon,\lambda,\gamma,|\Omega|} Z, \\ 
\sqrt{\lambda\gamma} p(1-\epsilon)  Y \le \frac{(t+1)^2 }{4t}  \gamma Y^{\frac{2t-p+1}{2(t+1)}} X^{\frac{p+1}{2(t+1)}} +  C_{\epsilon,\lambda,\gamma,|\Omega|} Z,
\end{array} 
 \end{eqnarray}
 where  $C_{\epsilon,\lambda,\gamma,|\Omega|} $ is a positive constant. 
 \end{lemma}

 \begin{proof} Let $t>1$ be a constant.  Multiply the second equation of $ (E)_{\lambda,\gamma}$ with $(1+u)^t-1$  and  integrate to get   
 \begin{equation*}
 \lambda \int_{\Omega} [(1+u)^t-1] (1+v)^p dx= \int_{\Omega} \mathcal L u [(1+u)^t-1] dx. 
 \end{equation*} 
 From Lemma \ref{fgprop}, we get
 \begin{equation*}
\int_{\Omega} \mathcal L u [(1+u)^t-1] dx =  \frac{1}{2} \int_{\mathbb R^n} \int_{\mathbb R^n}    \left[u(x)-u(z)  \right] \left[(1+u(x))^t-(1+u(z))^t \right] J(x-z) dx dz  . 
 \end{equation*} 
 Note that for $\alpha,\beta\in\mathbb R$, one can see that 
 \begin{equation*}
  [(1+\alpha)^t-(1+\beta)^t](\alpha-\beta) \ge \frac{4t}{(t+1)^2}  \left| (1+\alpha)^{\frac{t+1}{2}} - (1+\beta)^{\frac{t+1}{2}}  \right|^2 .  
 \end{equation*}
 Applying the above inequality for $\alpha=u(x)$ and $\beta=u(z)$, we obtain 
  \begin{equation*}
    [(1+u(x))^t-(1+u(z))^t][u(x)-u(z)] \ge \frac{4t}{(t+1)^2}  \left| (1+u(x))^{\frac{t+1}{2}} - (1+u(z))^{\frac{t+1}{2}}  \right|^2 .  
 \end{equation*}
From the above, we conclude 
 \begin{equation*}
 \lambda \int_{\Omega}  (1+u)^t (1+v)^p dx \ge \frac{4t}{(t+1)^2}  \frac{1}{2}  \int_{\mathbb R^n} \int_{\mathbb R^n} 
   \left| (1+u(x))^{\frac{t+1}{2}} - (1+u(z))^{\frac{t+1}{2}}  \right|^2 J(x-z) dx dz . 
 \end{equation*}  
Test the stability inequality, Corollary \ref{stablein},  on  $\zeta=(1+u)^{\frac{t+1}{2}}-1$ to get 
  \begin{eqnarray*}\label{}
&&\sqrt{\lambda\gamma} p \int_{\Omega} (1+u)^{\frac{p-1}{2}}  (1+v)^{\frac{p-1}{2}} [(1+u)^{\frac{t+1}{2}}-1]^2 dx 
\\&\le&  \frac{1}{2}\int_{\RR^n} \int_{\RR^n} \left| (1+u(x))^{\frac{t+1}{2}}- (1+u(z))^{\frac{t+1}{2}} \right|^2 J(x-z) dz dx , 
\end{eqnarray*}  
 Combining above inequalities, we conclude 
 \begin{equation}\label{rruv}
\sqrt{\lambda\gamma} p \int_{\Omega} (1+u)^{\frac{p-1}{2}}  (1+v)^{\frac{p-1}{2}} [(1+u)^{\frac{t+1}{2}}-1]^2 dx 
\le \frac{(t+1)^2}{4t}  \lambda \int_{\Omega}  (1+u)^t (1+v)^p dx . 
\end{equation}  
Expanding the left-hand side of the inequality and rearranging we get 
\begin{eqnarray}\label{l1e}
&&\sqrt{\lambda\gamma} p(1-\epsilon) \int_{\Omega} (1+u)^{\frac{p-1}{2}}  (1+v)^{\frac{p-1}{2}} (1+u)^{t+1} dx 
\\&&\nonumber \le \frac{(t+1)^2}{4t}  \lambda \int_{\Omega}  (1+u)^t (1+v)^p dx 
 +  \frac{ \sqrt{\lambda\gamma}  p}{\epsilon}   \int_{\Omega}  (1+u)^{\frac{p-1}{2}}  (1+v)^{\frac{p-1}{2}}  dx, 
\end{eqnarray}  
where we have used the inequality $a\le \frac{\epsilon}{2}a^2+\frac{1}{2\epsilon}$ for any $\epsilon>0$. From the H\"{o}lder's inequality we get 
\begin{eqnarray*}\label{}
&& \int_{\Omega}  (1+u)^t (1+v)^p dx 
\\& \le &\left[  \int_{\Omega} (1+u)^{\frac{p-1}{2}}  (1+v)^{\frac{p-1}{2}} (1+u)^{t+1} dx  \right]^{\frac{1}{\beta}} 
  \left[    \int_{\Omega} (1+v)^{\frac{p-1}{2}}  (1+u)^{\frac{p-1}{2}} (1+v)^{t+1} dx  \right]^{1-\frac{1}{\beta}} , 
\end{eqnarray*}  
where $\beta=\frac{2(t+1)}{2t-p+1}$. This and (\ref{l1e}) completes the proof of the first estimate in (\ref{lgXet}). Similarly, one can show the second  estimate.

 \end{proof} 

We now consider the MEMS system with singular power nonlinearities and  establish a relation between $\int_{\Omega} (1-u)^{\frac{1-p}{2}-t} (1-v)^{-\frac{p+1}{2}} dx$  and $\int_{\Omega} (1-v)^{\frac{1-p}{2}-t} (1-u)^{-\frac{p+1}{2}}  dx$ for some constant $t>{1}$.

  \begin{lemma}\label{lemuvmt} Under the same assumptions as Lemma \ref{lemuvm}, set 
 \begin{eqnarray*}
&& X:=\int_{\Omega} (1-u)^{-\frac{p+1}{2}-t+1} (1-v)^{-\frac{p+1}{2}} dx,  \ \ Y:=\int_{\Omega} (1-v)^{-\frac{p+1}{2}-t+1} (1-u)^{-\frac{p+1}{2}}  dx,
 \\&& Z:=\int_{\Omega} (1-u)^{-\frac{p+1}{2}}(1-v)^{-\frac{p+1}{2}} , 
 \end{eqnarray*}
 for $t>1$. Then, for some constant $0<\epsilon<1$ we get 
 \begin{eqnarray}\label{lgXmt}
 \begin{array}{lcl} 
\sqrt{\lambda\gamma} p(1-\epsilon) X \le \frac{(t-1)^2 }{4t}  \lambda X^{\frac{2t-p-1}{2(t-1)}} Y^{\frac{p-1}{2(t-1)}} +  C_{\epsilon,\lambda,\gamma,|\Omega|} Z, \\ 
\sqrt{\lambda\gamma} p(1-\epsilon)  Y \le \frac{(t-1)^2 }{4t}  \gamma Y^{\frac{2t-p-1}{2(t-1)}} X^{\frac{p-1}{2(t-1)}} +  C_{\epsilon,\lambda,\gamma,|\Omega|} Z,
\end{array} 
 \end{eqnarray}
 where $C_{\epsilon,\lambda,\gamma,|\Omega|} $ is a positive constant. 
 \end{lemma}
\begin{proof} Let $t>1$ be a constant.  Multiply the second equation of $ (M)_{\lambda,\gamma}$ with $(1-u)^{-t}-1$  and  integrate to get   
 \begin{equation*}
 \lambda \int_{\Omega} [(1-u)^{-t}-1] (1-v)^p dx= \int_{\Omega} \mathcal L u [(1-u)^{-t}-1] dx. 
 \end{equation*} 
 From Lemma \ref{fgprop}, we get
 \begin{equation*}
\int_{\Omega} \mathcal L u [(1-u)^{-t}-1] dx =  \frac{1}{2} \int_{\mathbb R^n} \int_{\mathbb R^n}    \left[u(x)-u(z)  \right] \left[(1-u(x))^{-t}-(1-u(z))^{-t} \right] J(x-z) dx dz . 
 \end{equation*} 
 Note that for $\alpha,\beta\in\mathbb R$, one can see that 
 \begin{equation*}
  [(1-\alpha)^{-t}-(1-\beta)^{-t}](\alpha-\beta) \ge \frac{4t}{(t-1)^2}  \left| (1-\alpha)^{\frac{-t+1}{2}} - (1-\beta)^{\frac{-t+1}{2}}  \right|^2 .  
 \end{equation*}
 Applying the above inequality for $\alpha=u(x)$ and $\beta=u(z)$, we obtain 
  \begin{equation*}
    [(1-u(x))^{-t}-(1-u(z))^{-t}][u(x)-u(z)] \ge \frac{4t}{(t-1)^2}  \left| (1-u(x))^{\frac{-t+1}{2}} - (1-u(z))^{\frac{-t+1}{2}}  \right|^2 .  
 \end{equation*}
From the above, we conclude 
 \begin{equation*}
 \lambda \int_{\Omega}  (1-u)^{-t} (1-v)^{-p} dx \ge \frac{4t}{(t-1)^2}  \frac{1}{2}  \int_{\mathbb R^n} \int_{\mathbb R^n} 
   \left| (1-u(x))^{\frac{-t+1}{2}} - (1-u(z))^{\frac{-t+1}{2}}  \right|^2 J(x-z) dx dz . 
 \end{equation*}  
Test the stability inequality, Corollary \ref{stablein},  on  $\zeta=(1-u)^{\frac{-t+1}{2}}-1$ to get 
  \begin{eqnarray*}\label{}
&&\sqrt{\lambda\gamma} p \int_{\Omega} (1-u)^{-\frac{p+1}{2}}  (1-v)^{-\frac{p+1}{2}} [(1-u)^{\frac{-t+1}{2}}-1]^2 dx 
\\&\le&  \frac{1}{2}\int_{\RR^n} \int_{\RR^n} \left| (1-u(x))^{\frac{-t+1}{2}}- (1-u(z))^{\frac{-t+1}{2}} \right|^2 J(x-z) dz dx . 
\end{eqnarray*}  
 Combining above inequalities, we conclude 
 \begin{equation}\label{rruv}
\sqrt{\lambda\gamma} p \int_{\Omega} (1-u)^{-\frac{p+1}{2}}  (1-v)^{-\frac{p+1}{2}} [(1-u)^{\frac{-t+1}{2}}-1]^2 dx 
\le \frac{(t-1)^2}{4t}  \lambda \int_{\Omega}  (1-u)^{-t} (1-v)^{-p} dx . 
\end{equation}  
Expanding the left-hand side of the inequality and rearranging we get 
\begin{eqnarray}\label{l1e}
&&\sqrt{\lambda\gamma} p(1-\epsilon) \int_{\Omega} (1-u)^{-\frac{p+1}{2}}  (1-v)^{-\frac{p+1}{2}} (1-u)^{-t+1} dx 
\\&&\nonumber \le \frac{(t-1)^2}{4t}  \lambda \int_{\Omega}  (1-u)^{-t} (1-v)^{-p} dx 
 +  \frac{ \sqrt{\lambda\gamma}  p}{\epsilon}   \int_{\Omega}  (1-u)^{-\frac{p+1}{2}}  (1-v)^{-\frac{p+1}{2}}  dx, 
\end{eqnarray}  
where we have used the inequality $a\le \frac{\epsilon}{2}a^2+\frac{1}{2\epsilon}$ for any $\epsilon>0$ and $a\in\mathbb R$. From the H\"{o}lder's inequality we get 
\begin{eqnarray*}\label{}
&& \int_{\Omega}  (1-u)^{-t} (1-v)^{-p} dx 
\\& \le &\left[  \int_{\Omega} (1-u)^{-\frac{p+1}{2}}  (1-v)^{-\frac{p+1}{2}} (1-u)^{-t+1} dx  \right]^{\frac{1}{\beta}} 
  \left[    \int_{\Omega} (1-v)^{-\frac{p+1}{2}}  (1-u)^{-\frac{p+1}{2}} (1-v)^{-t+1} dx  \right]^{1-\frac{1}{\beta}} , 
\end{eqnarray*}  
where $\beta=\frac{2(t-1)}{2t-p-1}$. This and (\ref{l1e}) completes the proof of the first estimate in (\ref{lgXet}). Similarly, one can show the second  estimate.

 \end{proof}

In regards to the gradient system with superlinear nonlinearities satisfying (\ref{R}) we establish    an integral estimate that yields $L^2(\Omega)$ of the function $f'(u)g'(v)$. We then use this to conclude estimates on the nonlinearities of the gradient system.  Our methods and ideas in the proof are inspired by the ones developed in \cite{cf} and originally by Nedev in \cite{Nedev}.

\begin{lemma}\label{lemab}
Suppose  that $f$ and
$g$ both satisfy condition (\ref{R}) and $a:=f'(0)>0 $ and $ b:=g'(0)>0$. Assume that $ f',g'$ are convex and (\ref{deltaeps}) holds.   Let $ (\lambda^*,\gamma^*) \in \Upsilon$ and $ (u,v)$ denote the extremal solution associated with $ (H)_{\lambda^*,\gamma^*}$. 
Then,  there exists a positive constant $  C < \infty$ such that
\begin{equation*}
\int_{\Omega} f'(u) g'(v) (f'(u)-a) (g'(v)-b) \le C.
\end{equation*}
\end{lemma}

\begin{proof} We obtain uniform estimates for any minimal solution $(u,v)$ of $(H)_{\lambda,\gamma}$ on the ray $ \Gamma_\sigma$ and then one sends $ \lambda \nearrow \lambda^*$ to obtain the same estimate for $ (u^*,v^*)$.   Let $(u,v)$ denote a smooth minimal solution of $(H)_{\lambda,\gamma}$ on the ray $\Gamma_\sigma$ and  put $ \zeta:= f'(u)-a$ and $ \eta:=g'(v)-b$ into (\ref{stabilityH}) to obtain
\begin{eqnarray*}\label{}
&&\int_{\Omega} \left[f''(u) g(v) (f'(u)-a)^2 +  f(u) g''(v) (g'(v)-b)^2 + 2  f'(u) g'(v) (f'(u)-a)(g'(v)-b)\right]dx
 \\ &\le&\nonumber
  \frac{1}{2} \int_{\RR^n} \int_{\RR^n}\left(  \frac{1}{\lambda} |f'(u(x))-  f'(u(z))|^2 + \frac{1}{\gamma}  |g'(v(x))- g'(v(z))|^2 \right)  J(x-z)  dz dx  . 
\end{eqnarray*}
Note that for all $\alpha,\beta\in\mathbb R$, one can see that
\begin{equation*}\label{}
|f'(\beta)-f'(\alpha)|^2=\left|\int_\alpha^\beta f''(s) ds\right|^2 \le \int_\alpha^\beta |f''(s)|^2 ds (\beta-\alpha)= (h_1(\beta)-h_1(\alpha))(\beta-\alpha), 
\end{equation*}
when $h_1(s):=\int_0^s  |f''(w)|^2 dw$. Similar inequality holds for the function $g$ 
that is 
\begin{equation*}\label{}
|g'(\beta)-g'(\alpha)|^2 \le  (h_2(\beta)-h_2(\alpha))(\beta-\alpha), 
\end{equation*}
when $h_2(s):=\int_0^s  |g''(w)|^2 dw$. Set $\beta=u(x)$ and $\alpha=u(z)$ and $\beta=v(x)$ and $\alpha=v(z)$ in the above inequalities to conclude 
\begin{eqnarray*}\label{}
&&\frac{1}{2} \int_{\RR^n} \int_{\RR^n}\left(  \frac{1}{\lambda} |f'(u(x))-  f'(u(z))|^2 + \frac{1}{\gamma}  |g'(v(x))- g'(v(z))|^2 \right)  J(x-z)  dz dx
\\&\le&
 \frac{1}{2} \int_{\RR^n} \int_{\RR^n}  \frac{1}{\lambda} [h_1(u(x))-  h_1(u(z))][u(x)-  u(z)]  J(x-z) dz dx
 \\&&+   
\frac{1}{2} \int_{\RR^n} \int_{\RR^n} \frac{1}{\gamma}  [h_2(v(x))-  h_2(v(z))][v(x)-  v(z)]   J(x-z)  dz dx. 
\end{eqnarray*}
From the equation of system and Lemma \ref{fgprop}, we get 
\begin{eqnarray*}\label{}
&&\frac{1}{2} \int_{\RR^n} \int_{\RR^n}\left(  \frac{1}{\lambda} |f'(u(x)) -  f'(u(z))|^2 + \frac{1}{\gamma}  |g'(v(x))- g'(v(z))|^2 \right)  J(x-z)  dz dx
\\&\le& 
\int_{\Omega} \left[ h_1(u) \frac{1}{\lambda} \mathcal L(u)  +  h_2(v) \frac{1}{\gamma} \mathcal L(v) \right]dx  
\\&=& 
\int_{\Omega} \left[ h_1(u)f'(u)g(v) dx + h_2(v)f(u)g'(v) \right]dx . 
\end{eqnarray*}
From this and we conclude that 
\begin{eqnarray}\label{stafg}
&&\int_{\Omega} \left[f''(u) g(v) (f'(u)-a)^2 +  f(u) g''(v) (g'(v)-b)^2 + 2  f'(u) g'(v) (f'(u)-a)(g'(v)-b)\right]dx
 \\ &\le&\nonumber
\int_{\Omega} \left[ h_1(u)f'(u)g(v) dx + h_2(v)f(u)g'(v) \right] dx . 
\end{eqnarray}
 Given the assumptions, there is some $ M>1$ large enough  and $0<\delta<1$ that for all $ u \ge M$ we have 
 $ h_1(u) \le \delta f''(u)(f'(u)-a)$ for all $ u \ge M$.  Then,  we have
  \begin{equation*} 
   \int_{\Omega} h_1(u) g(v) f'(u) = \int_{u \ge M} + \int_{u <M} \le  \delta \int f''(u) g(v) (f'(u)-a)^2  + \int_{u<M} \int  h_1(u) g(v) f'(u) .
  \end{equation*}
     We now estimate the last integral in the above.    Let $ k \ge 1$ denote a natural number. Then, 
   \begin{equation*} 
   \int_{u<M}  h_1(u) g(v) f'(u)  = \int_{u<M, v<kM} + \int_{u<M, v \ge kM} = C(k,M) + \int_{u<M, v \ge kM}h_1(u) g(v) f'(u) . 
    \end{equation*}
     Note that this integral is bounded above by
  \begin{equation*}
   \sup_{u<M} \frac{h_1(u)}{(f'(u)-a)} \sup_{v >kM} \frac{g(v)}{(g'(v)-b) g'(v)} \int (f'(u)-a) (g'(v)-b) f'(u) g'(v).
   \end{equation*}
 From the above estimates, we conclude that for sufficiently large $ M$ and for all $1 \le k$ there is some positive constant $ C(k,M)$ and $0<\delta <1$ that
\begin{eqnarray}\label{h1gf}
\ \ \ \ \  \int_{\Omega} h_1(u) g(v) f'(u)  &\le &  \delta \int f''(u) g(v) (f'(u)-a)^2  
 +  C(k,M)  
 \\&& \nonumber+  \sup_{u<M} \frac{h_1(u)}{(f'(u)-a)} \sup_{v >kM} \frac{g(v)}{(g'(v)-b) g'(v)} \int (f'(u)-a) (g'(v)-b) f'(u) g'(v).
 \end{eqnarray}
 Applying the same argument,  one can show that for sufficiently large $M$ and for all $ 1 \le k$ there is some positive constant $ D(k,M)$ and $0<\epsilon <1$ that
  \begin{eqnarray}\label{h2gf}
\ \ \ \ \   \int_{\Omega} h_2(v) g'(v) f(u)  &\le &  \epsilon \int f(u) g''(v) (g'(v)-b)^2  
 +  D(k,M)  
 \\&& \nonumber +  \sup_{v<M} \frac{h_2(v)}{(g'(v)-b)} \sup_{u >kM} \frac{f(u)}{(f'(u)-a) f'(u)} \int (f'(u)-a) (g'(v)-b) f'(u) g'(v).
 \end{eqnarray}
 Note that $ f''(u),g''(v) \rightarrow \infty$ when $u,v\to\infty$. This implies that 
   \begin{equation*}
    \lim_{k \rightarrow \infty}  \sup_{u >kM} \frac{f(u)}{(f'(u)-a) f'(u)} =0 \ \ \text{and} \ \   \lim_{k \rightarrow \infty}  \sup_{v >kM} \frac{g(v)}{(g'(v)-b) g'(v)} =0  . 
    \end{equation*}
 Now set $ k$ to be sufficiently large and substitute (\ref{h1gf}) and (\ref{h2gf}) in  (\ref{stafg}) to 
    complete the proof.  Note that  see that all the integrals in (\ref{stafg}) are bounded independent of $ \lambda$ and $\gamma$.

\end{proof}

\section{Regularity of the extremal solution; Proof of Theorem \ref{thmg}-\ref{nedev}}\label{secreg}

In this section, we apply the integral estimates established in the latter section to prove regularity results for extremal solutions of systems mentioned in the introduction earlier. 
\\
 \\
 \noindent {\it Proof of Theorem \ref{thmg}}. 
  We shall provide the proof for the case of $n>2s$, since otherwise is straightforward.  Let $(u,v)$ be the smooth minimal solution of  $ (G)_{\lambda,\gamma}$ for $\frac{\lambda^*}{2}<\lambda<\lambda^*$ and $\frac{\gamma^*}{2}<\gamma<\gamma^*$.   From Lemma \ref{lemuvg} we conclude that  
 \begin{equation}
 \int_{\Omega} e^{u+v} dx \le C_{\lambda,\gamma,|\Omega|} . 
  \end{equation}
 From this and Lemma \ref{lemuvgt},  we conclude that for $t>\frac{1}{2}$
  \begin{equation}
\lambda\gamma \left[1-\left(\frac{t }{4}+\epsilon\right)^2\right] XY \le  C_{\epsilon,\lambda,\gamma,|\Omega|} 
\left(1+X^{\frac{2t-1}{2t}} Y^{\frac{1}{2t}}  + Y^{\frac{2t-1}{2t}} X^{\frac{1}{2t}}\right) . 
  \end{equation}
Therefore, for every $t<4$  either $X$ or $Y$ must be
bounded where $X$ and $Y$ are given by 
 \begin{equation}
 X:=\int_{\Omega} e^{\frac{2t+1}{2}u} e^{\frac{v}{2}} dx \ \ \text{and} \ \ Y:=\int_{\Omega} e^{\frac{2t+1}{2}v} e^{\frac{u}{2}} dx. 
 \end{equation}
Without loss of generality, assume that $\lambda\le \gamma$ implies that $u\le v$ and therefore $e^u$ is bounded in $L^{q}(\Omega)$ for $q=t+1<5$. Therefore, in light of Proposition \ref{propregL} we have $u\in L^\infty(\Omega)$ for $\frac
{n}{2s}<5$ that is $n<10s$. 

Now, let $(u,v)$ be the smooth minimal solution of  $ (E)_{\lambda,\gamma}$ for $\frac{\lambda^*}{2}<\lambda<\lambda^*$ and $\frac{\gamma^*}{2}<\gamma<\gamma^*$.   From Lemma \ref{lemuve} we conclude that  
 \begin{equation}\label{}
 \int_{\Omega} (1+u)^p(1+v)^p dx \le C_{\lambda,\gamma,|\Omega|} . 
 \end{equation}
From this and Lemma \ref{lemuvet},  we conclude that 
  \begin{equation}
\lambda\gamma \left[p^2(1-\epsilon)^2 -\left( \frac{(t+1)^2 }{4t}  \right)^2\right] XY \le  C_{\epsilon,\lambda,\gamma,|\Omega|} 
\left(1+X^{\frac{2t-p+1}{2(t+1)}} Y^{\frac{p+1}{2(t+1)}}  + Y^{\frac{2t-p+1}{2(t+1)}} X^{\frac{p+1}{2(t+1)}}  \right) . 
  \end{equation}
Therefore, for every $1\le t< 2p+2\sqrt{p(p-1)}-1$  either $X$ or $Y$ must be
bounded where $X$ and $Y$ are given by 
 \begin{equation}
 X:=\int_{\Omega} (1+u)^{\frac{p-1}{2}+t+1} (1+v)^{\frac{p-1}{2}} dx \ \ \text{and} \ \ Y:=\int_{\Omega} (1+v)^{\frac{p-1}{2}+t+1} (1+u)^{\frac{p-1}{2}}  dx. 
 \end{equation}
 Without loss of generality, assume that $\lambda\le \gamma$ implies that $u\le v$ and therefore $(1+u)$ is bounded in $L^{q}(\Omega)$ for $q=p+t$. We now rewrite the system $ (E)_{\lambda,\gamma}$ for the extremal solution $(u,v)$ that is 
   \begin{eqnarray*}
  \left\{ \begin{array}{lcl}
\hfill  \mathcal L u    &=&\lambda^*   c_{v}(x) v + \lambda^* \qquad \text{in} \ \  \Omega,   \\
\hfill  \mathcal L v  &=& \gamma^* c_{u}(x) u +\gamma^*  \qquad \text{in} \ \  \Omega,  \\
\end{array}\right.
  \end{eqnarray*}  
  when   $0\le c_{v}(x)= \frac{(1+v)^{p}-1}{v} \le p v^{p-1}$ and $0\le c_{u}(x)= \frac{(1+u)^{p}-1}{u} \le p u^{p-1}$ where convexity argument is applied. From the regularity theory, Proposition \ref{propregL}, we conclude that $v\in L^\infty(\Omega)$ provided $ c_{u}(x)\in L^{r}(\Omega)$ when $r>\frac{n}{2s}$.   This implies that $\frac{n}{2s} <\frac{p+t}{p-1}$ for $1\le t<2p+2\sqrt{p(p-1)}-1$. This completes the proof for the case of Lane-Emden system $ (E)_{\lambda,\gamma}$. The proof for the case of $ (M)_{\lambda,\gamma}$ is very similar and replies on applying Lemma \ref{lemuvm} and Lemma \ref{lemuvmt}.


 \hfill $\Box$

\begin{remark}\label{rem1}
 Even though the above theorem is optimal as $s\to 1$, it is not optimal for smaller values of $0<s<1$.
 \end{remark}
  In this regard, consider the case of $\lambda=\gamma$ and the Gelfand system turns into ${(-\Delta)}^s u= \lambda e^u$ in the entire space $\mathbb R^n$. It is known that the explicit singular solution $u^*(x)=\log \frac{1}{|x|^{2s}}$ is stable solution of the scalar Gelfand equation   if and only if 
 \begin{equation*}
\frac{ \Gamma(\frac{n}{2}) \Gamma(1+s)}{\Gamma(\frac{n-2s}{2})} \le 
\frac{ \Gamma^2(\frac{n+2s}{4})}{\Gamma^2(\frac{n-2s}{4})} , 
\end{equation*}
for the constant 
$$\lambda=2^{2s}\frac{ \Gamma(\frac{n}{2}) \Gamma(1+s)}{\Gamma(\frac{n-2s}{2})}  . 
$$
This implies that the extremal solution of the fractional Gelfand equation should be regular for 
\begin{equation*}
\frac{ \Gamma(\frac{n}{2}) \Gamma(1+s)}{\Gamma(\frac{n-2s}{2})} >
\frac{ \Gamma^2(\frac{n+2s}{4})}{\Gamma^2(\frac{n-2s}{4})} . 
\end{equation*}
In particular, the extremal solution should be bounded when $n\le 7$ for $0<s<1$.  We refer interested to \cite{rs,r1} for more details. Now, consider the case of Lane-Emden equation $(-\Delta)^s u= \lambda u^p$ in the entire space $\mathbb R^n$. It is also known that the explicit singular solution $u_s(x) = A |x|^{-\frac{2s}{p-1}}$ where the constant $A$ is given by 
\begin{equation*}
A^{p-1} =\frac{\Gamma(\frac{n}{2}-\frac{s}{p-1}) \Gamma(s+\frac{s}{p-1})}{\Gamma(\frac{s}{p-1}) \Gamma(\frac{n-2s}{2}-\frac{s}{p-1})}  ,
\end{equation*}
 is a stable solution of the scalar Lane-Emden equation if and only if 
\begin{equation*}\label{} 
p \frac{\Gamma(\frac{n}{2}-\frac{s}{p-1}) \Gamma(s+\frac{s}{p-1})}{\Gamma(\frac{s}{p-1}) \Gamma(\frac{n-2s}{2}-\frac{s}{p-1})}
\le 
\frac{ \Gamma^2(\frac{n+2s}{4}) }{\Gamma^2(\frac{n-2s}{4})}. 
\end{equation*}
This yields that the extremal solution of the above equation should be regular for 
\begin{equation*}\label{} 
p \frac{\Gamma(\frac{n}{2}-\frac{s}{p-1}) \Gamma(s+\frac{s}{p-1})}{\Gamma(\frac{s}{p-1}) \Gamma(\frac{n-2s}{2}-\frac{s}{p-1})}
> 
\frac{ \Gamma^2(\frac{n+2s}{4}) }{\Gamma^2(\frac{n-2s}{4})}. 
\end{equation*}
As $s\to 1$, the above inequality is consistent with the dimensions given in (\ref{dime}). For more information, we refer interested readers to Wei and the author in \cite{fw} and Davila et al. in \cite{ddw}. Given above, proof of the optimal dimension for regularity of extremal solutions remains an open problem. 

We now provide a proof for Theorem \ref{nedev} that deals with a regularity result for the gradient system  $ (H)_{\lambda,\gamma}$ with  general nonlinearities $f$ and $g$.  Note that for the case of local scalar equations such results are provided by Nedev in \cite{Nedev} for $n=3$ and Cabr\'{e} in \cite{Cabre} for $n=4$.  For the case of scalar equation with the fractional Laplacian operator, Ros-Oton and  Serra in \cite{rs} established regularity results for dimensions $n<4s$ when $0<s<1$.  For the case of local gradient systems, that is when $s=1$, such a regularity  result is established by the author and Cowan in \cite{cf} in dimensions $n\le 3$. 
\\
\\
\noindent {\it Proof of Theorem \ref{nedev}}. 
We suppose that $ (\lambda^*,\gamma^*) \in \Upsilon$ and $(u,v)$ is the associated extremal solution of $(G)_{\lambda^*,\gamma^*}$.  Set $ \sigma=\frac{\gamma^*}{\lambda^*}$.
From Lemma \ref{lemab}, we conclude that $ f'(u) g'(v) \in L^2(\Omega)$.   Note that this and the convexity of $ g$ show that
\begin{equation}
 \int_\Omega \frac{f'(u)^2 g(v)^2}{(v+1)^2} \le C.
 \end{equation}
Note that    $ (-\Delta)^s u\in L^1$  and $ (-\Delta)^s v\in L^1$ and hence  we have $ u,v \in L^{p}$ for any $ p <\frac{n}{n-2s}$.   We now use the domain decomposition method as in \cite{Nedev,cf}.  Set
\begin{eqnarray*}
 \Omega_1 &:=& \left\{ x: \frac{f'(u)^2 g(v)^2}{(v+1)^2} \ge f'(u)^{2-\alpha} g(v)^{2-\alpha} \right\},\\
\Omega_2 &:=& \Omega \backslash \Omega_1 = \left\{ x : f'(u) g(v) \le (v+1)^\frac{2}{\alpha} \right\},
 \end{eqnarray*}
where $ \alpha  $ is a positive constant and will be fixed later.  First note that
\begin{equation}
 \int_{\Omega_1} (f'(u) g(v))^{2-\alpha} \le \int_\Omega \frac{f'(u)^2 g(v)^2}{(v+1)^2} \le C.
 \end{equation}
  Similarly we have
\begin{equation}
 \int_{\Omega_2} (f'(u) g(v))^p \le \int_\Omega (v+1)^\frac{2p}{\alpha}.
  \end{equation}
We shall consider the case of $n>2s$, and $n\le 2s$ is straightforward as discussed in Section \ref{secpre}.  Taking $ \alpha= \frac{4(n-2s)}{3n-4s}$ and using the    $ L^{p}$-bound  on $ v$  for $p<\frac{n}{n-2s}$ shows that $ f'(u) g(v) \in  L^{p}(\Omega)$ for $p<\frac{2n}{3n-4s}$.   By a symmetry argument we also have $ f(u) g'(v) \in   L^{p}(\Omega)$ for $p<\frac{2n}{3n-4s}$. Therefore, $(-\Delta )^s u,(-\Delta )^s v\in L^p(\Omega)$ when $p<\frac{2n}{3n-4s}$. From elliptic estimates we conclude that $u,v\in L^p(\Omega)$ when  $p<\frac{2n}{3n-8s}$ for $n>\frac{8}{3}s$ and when $p<\infty$ for $n=\frac{8}{3}s$ and when  $p\le \infty$ for $n<\frac{8}{3}s$.  This completes the proof when $2s\le n<\frac{8}{3}s$. 
\\
Now, set  $ \alpha= \frac{3n-8s}{2(n-2s)}$. From the above estimate $u,v\in L^p(\Omega)$ when  $p<\frac{2n}{3n-8s}$ for $n>\frac{8}{3}s$ on $v$ and domain decomposition arguments, we get $ f(u) g'(v),  f'(u) g(v)\in   L^{p}(\Omega)$ for $p<\frac{n}{2(n-2s)}$. Therefore, $(-\Delta )^s u,(-\Delta )^s v\in L^p(\Omega)$ with the latter bounds for $p$. From elliptic estimates we get 
$u,v\in L^p(\Omega)$ when  $p<\frac{n}{2(n-3s)}$ for $n>3s$ and when $p<\infty$ for $n=3s$ and when  $p\le \infty$ for $n<3s$. We perform the above arguments once more to arrive at 
$u,v\in L^p(\Omega)$ when  $p<\frac{2n}{5n-16s}$ for $n>\frac{16}{5}s$ and when $p<\infty$ for $n=\frac{16}{5}s$ and when  $p\le \infty$ for $n<\frac{16}{5}s$. This completes the proof for 
$\frac{8}{3}s\le n < \frac{16}{5}s$ containing $n=3s$.
\\
 Now, suppose that   $u,v\in L^p(\Omega)$ for $p<p_*$. Then, notice that 
\begin{equation*}
 \int_{\Omega_2} (f'(u) g(v))^p \le \int_\Omega (v+1)^\frac{2p}{\alpha} \le C \ \ \text{when} \ \ p<\frac{\alpha p_*}{2}. 
  \end{equation*}
 Set  $ \alpha= \frac{4}{p_*+2}$. Then, from the above we conclude that $ f'(u) g(v) \in  L^{p}(\Omega)$ for $p<\frac{2p_*}{p_*+2}$ and similarly $ f(u) g'(v) \in  L^{p}(\Omega)$ for $p<\frac{2p_*}{p_*+2}$. Applying the fact that $(-\Delta)^s u,(-\Delta)^s v\in L^p(\Omega)$ for the same range of $p$. From elliptic estimates we conclude that 
 \begin{equation*}
 u,v\in L^{p}(\Omega) \ \ \text{when} \ \ p<\frac{2p_*n}{p_*( n-4s) +2n} . 
 \end{equation*}
Applying the above elliptic estimates arguments, we conclude the boundedness of solutions when $n<4s$.

\hfill $\Box$ 


We end this section with power polynomial nonlinearities for the gradient system $ (H)_{\lambda^*,\gamma^*}$ and we provide regularity of the extremal solution. For the case of local systems, that is when $s=1$,  a similar result is given in \cite{cf}. Due to the technicality of the proof we omit it here.    
 
\begin{thm} \label{grade} Let $ f(u)=(1+u)^p $ and $  g(v)=(1+v)^q$ when $ p,q>2$.  Assume that $(\lambda^*,\gamma^*) \in \Upsilon$.  Then,  the associated extremal solution of  $ (H)_{\lambda^*,\gamma^*}$ is bounded provided
\begin{equation} \label{gradp} 
n < 2s + \frac{4s}{p+q-2} \max\{ T(p-1), T(q-1)\} , 
\end{equation}
when $ T(t):=  t+ \sqrt{t(t-1)}$.  
\end{thm}

\vspace*{.4 cm }

\noindent {\it Acknowledgment}.   The  author would like to thank Professor Xavier Ros-Oton for online communications and comments in regards to Section \ref{secpre}. The author is grateful to Professor Xavier Cabr\'{e} for bringing reference \cite{sp} to his attention and for the comments in regards to Section \ref{secin}.

\end{document}